\documentclass[12pt,twoside,a4paper]{amsart}

\usepackage{amssymb}
\usepackage{latexsym}
\usepackage{amsmath}
\usepackage{euscript}
\usepackage{graphics}

\newcommand{\tA}{\widetilde{A}}

\newcommand{\gH}{{\mathfrak{H}}}
\newcommand{\gN}{{\mathfrak{N}}}

\def\a{\alpha} \def\b{\beta}

\def\C{\mathbb C}
\def\R{\mathbb R}

\def\B{\mbox{\boldmath$B$}}

\newcommand{\D}{{\mathbb{D}}}

\newcommand{\eps}{\varepsilon}
\newcommand{\ff}{\varphi}

\newcommand{\cH}{{\mathcal{H}}}
\newcommand{\cD}{{\mathcal{D}}}

\newcommand{\cI}{{\mathcal{I}}}
\def\l {\lambda} \def\G{\Gamma} \def\g{\gamma}
\def\cd{\cdot} \def\up{\upharpoonright} \def\th{\Theta}
\def\bta{\{\cH_0\oplus \cH_1,\Gamma ^B,\Gamma ^A\}}

\def\wt#1{{{\widetilde #1} }}
\def\wh#1{{{\,\widehat #1\,} }}
\def\ov {\overline}

\def\bm\chi{\mbox{\boldmath$\chi$}}

\def\Ext{{\rm Ext\,}}

\def\ran{{\rm ran\,}}
\def\dom{{\rm dom\,}}

\def\dim{{\rm dim\,}}
\def\loc{{\rm loc}}

\def\rank{{\rm rank\,}}

\let\xker=\ker \def\ker{{\xker\,}}

\def\supp{{\rm supp\,}}

\def\Dma{D_{\rm{max}}} \def\Dmi{D_{{\rm min}}} \def\Dtma{D_{\th}} \def\Dtmi{D_{\th,0}} \def\Dt{D_\th}

\DeclareMathOperator{\im}{Im}

\newtheorem{theorem}{Theorem}[section]
\newtheorem{proposition}[theorem]{Proposition}
\newtheorem{corollary}[theorem]{Corollary}
\newtheorem{lemma}[theorem]{Lemma}
\newtheorem{definition}[theorem]{Definition}

\theoremstyle{definition}

\newtheorem{remark}[theorem]{Remark}

\numberwithin{equation}{section}

\newcommand{\be}{\begin{equation}}
\newcommand{\ee}{\end{equation}}
\newcommand{\beq}{\begin{equation*}}
\newcommand{\eeq}{\end{equation*}}
\newcommand{\ben}{\begin{eqnarray}}
\newcommand{\een}{\end{eqnarray}}
\newcommand{\bea}{\begin{eqnarray*}}
\newcommand{\eea}{\end{eqnarray*}}

\begin{document}
\title[J-selfadjointness ]
{Weyl Solutions and J-selfadjointness for Dirac operators }
\author{B.~Malcolm~Brown}
\author{Martin~Klaus}
\author{Mark~Malamud}
\author{Vadim~Mogilevskii}
\author{Ian~Wood}
\address{
Cardiff School of Computer Science   and  Informatics\\ Cardiff University \\ Queen's Buildings\\ 5 The Parade\\ Cardiff CF24 3AA\\ UK}
\email{Malcolm.Brown@cs.cardiff.ac.uk}
\address{Mathematics Department,
Virginia Tech\\
Blacksburg, VA 24061-0123}
\email{klaus@math.vt.edu}
\address{ Peoples' Friendship University of Russia (RUDN University)\\
6 Miklukho-Maklaya St, Moscow, 117198, Russian Federation}
   \email{malamud3m@gmail.com}
\address{Poltava V.G. Korolenko National Pedagogical University\\Ostrogradski Street 2\\
36000 Poltava\\ Ukraine}
\email{vadim.mogilevskii@gmail.com}
\address{
School of Mathematics, Statistics and Actuarial Science\\ University of Kent\\ Sibson Building\\ Canterbury CT2 7FS\\ UK}
\email{i.wood@kent.ac.uk}
%\curaddr{}
%\dedicatory{}
%\date{\today}
%\thanks{The first and the second author were supported by
%the Academy of Finland (projects 40362, 49349, 79774).}
%\translator{}


%      \begin{abstract}
%A description of the set of $m$-sectorial extensions of a dual pair
%$\{A_1,A_2\}$ of nonnegative operators is obtained.
%Some classes of nonaccretive extensions of the dual pair $\{A_1,A_2\}$
%are described too. Both problems are reduced to similar problems for
%a dual pair $\{T_1,T_2\}$ of nondensely defined symmetric contractions
%$T_j=(I-A_j)(I+A_j)^{-1},\ j\in\{1,2\}$.
%In turn these problems are reduced to the investigation of
%the corresponding operator "holes".
%A complete description of the set of all proper and improper extensions
%of a nonnegative operator is obtained too.
%    \end{abstract}
%%%%%%%%%%%%%%%%%%%%%%%%%%%%%%%%%%%%%%%%%%%%%%%%%%%%%%%%%%%%%%%%%%%%%%%%%%%%%%%

%%%\newcommand\qqq{}  % for strange places ;-)

%%\DeclareMathOperator\re{Re}
%%\DeclareMathOperator\im{Im}
%%%\DeclareMathOperator\Ext{Ext}
%%%\DeclareMathOperator*\slim{s-lim}

%%\newcommand\limi{\lim\limits_{n\to\infty}}
%%\newcommand\liml{\lim\limits}
%%\newcommand\dup{\{T_1,T_2\}}
%%\newcommand\CT{\Ext_{T_1}(-1,1)}
%%\newcommand\SA{\Ext_A(0,\infty)}
%%\newcommand\ExtA{\Ext_A((0,\infty);\ff)}
%%\newcommand\bA{\begin{pmatrix} A_{11}& A_{12}\\ A_{21}& A_{22} \end{pmatrix}}
%%\newcommand\HH[2]{[\gH_{#1},\,\gH_{#2}]}

%%%%%%%%%%%%%%%%%%%%%%%%%%%%%%%%%%%%%%%%%%%%%%%%%%%%%%%%%%%%%%%%%%%%%%%%%%%%%%%
%%\By{{\sc  M. M. Malamud}}
%%%\Names{Malamud}
%%%\Email{mmm@univ.donetsk.ua}

     \begin{abstract}
We consider a non-selfadjoint Dirac-type differential expression
        \begin{equation}\label{Abstr_1.1}
D(Q)y:= J_n \frac{dy}{dx} + Q(x)y,
        \end{equation}
with a  non-selfadjoint potential  matrix  $Q \in L^1_{\loc}({\mathcal I},{\C}^{n\times
n})$ and  a signature matrix  $J_n =-J_n^{-1} = -J_n^*\in \C^{n\times n}$. Here
${\mathcal I}$  denotes either the line $\R$ or the half-line $\R_+$. With this
differential expression one associates in $L^2(\mathcal I,\C^{n})$  the (closed) maximal
and minimal operators $D_{\max}(Q)$ and $D_{\min}(Q)$, respectively. One of our main
results for the whole line case states that $D_{\max}(Q) = D_{\min}(Q)$ in $L^2(\R,\C^{n})$. Moreover, we show
that if the minimal operator $D_{\min}(Q)$ in $L^2(\R,\C^{n})$  is $j$-symmetric with
respect to an appropriate involution $j$, then  it is $j$-selfadjoint.

Similar results are valid in the case of  the semiaxis $\R_+$. In particular, we show
that if $n=2p$ and the minimal  operator $D^+_{\min}(Q)$ in $L^2(\R_+,\C^{2p})$  is
$j$-symmetric, then there exists a $2p\times p$-Weyl-type matrix solution $\Psi(z,
\cdot)\in L^2(\R_+,\mathbb{C}^{2p\times p})$ of the equation $D^+_{\max}(Q)\Psi(z,
\cdot)= z\Psi(z, \cdot)$. A similar result is valid for the expression \eqref{Abstr_1.1}
with a potential matrix having a bounded imaginary part. This leads to the existence of
a unique Weyl function for   the expression \eqref{Abstr_1.1}. The main
results are proven by means of  a reduction to the self-adjoint case  by using the
technique  of dual pairs of operators.
%% on $\R_+.$

The differential expression  \eqref{Abstr_1.1} is of significance  as it appears in the
Lax formulation of the vector-valued  nonlinear   Schr{\"o}dinger equation.
   \end{abstract}

\maketitle

%%%{\noindent}

\bigskip

%%%%%%%%%%%%%%%%%%%%%%%%%%%%%%%%%%%%%%%%%%%%%%%%%%%%%%%%%%%%%%%%%%%%%%%%%%%%%%%
\section{Introduction} \label{intro}

Let $\mathcal I$ be either the half-line $\R_+=[0,\infty)$ or the whole line $\R$. The
main object of the  paper is  a study of the   general (not necessarily formally self-adjoint) Dirac-type differential  expression
%%operators
%
        \begin{equation}\label{Intro_1.1}
D(Q)y:= J_n \frac{dy}{dx} + Q(x)y,
        \end{equation}
%
%
%%in $L^2(\R,\C^{n})$
with a  potential  matrix  $Q \in L^1_{\loc}({\mathcal I},{\C}^{n\times n})$ and  a
signature matrix  $J_n =-J_n^{-1} = -J_n^*\in \C^{n\times n}$. One   may associate  with this differential
expression  in $L^2(\mathcal I,\C^{n})$  the (closed) maximal  and minimal
operators $D_{\max}(Q)$ and $D_{\min}(Q)$, respectively. The maximal operator
$\Dma(Q)$ is defined by
     \begin{gather*}
\dom \Dma(Q) = \{y\in L^2(\mathcal I,\C^{n}): y\in {AC_{\loc}}(\mathcal I,\C^{n}),\  D(Q)y \in L^2(\mathcal I,\C^{n})\}, \\
 \Dma(Q) y=D (Q)y, \;\;\;y\in\dom\Dma(Q).
     \end{gather*}
%
%%Moreover, let $\Dmi':=\Dma\up \dom \Dmi'$, where $\dom \Dmi'$  %%%\up \dom \Dmi'$
The minimal operator $D_{\min}(Q)$ is  the closure of the operator $\Dmi'(Q)$, where $\Dmi'(Q)$ is the
restriction of $\Dma(Q)$   to the set of functions $y\in\dom\Dma(Q)$  with compact support
%%${\rm supp\,}y$
(in the case of the half-line  we mean that     ${\rm supp\,}y \subset (0,\infty))$. Clearly, both  the
operators $ \Dma(Q)$ and $D_{\min}(Q)$  are densely defined.

Assume that $\mathcal I=\R$ and   that the expression \eqref{Intro_1.1} is formally selfadjoint,
that is $Q(x)=Q^*(x)$ (a.e.~on $\R$). Then according to \cite[Sect. 8.6]{LS91}  and
\cite[Theorem 3.2]{LM03} we have $D_{\min}(Q) =D_{\max}(Q)(= D_{\min}(Q)^*)$, i.e., the
minimal  operator is self-adjoint and coincides with the maximal one.

To distinguish between  operators acting on the
half-line $\R_+$  from those acting on the  line   $\R$, we write $D^+_{\max}(Q)$ and $D^+_{\min}(Q)$ instead of $D_{\max}(Q)$
and $D_{\min}(Q)$, respectively. With this notation, in the case of the half-line the
analogue of our previous result     reads as follows: the minimal operator
$D^+_{\min}(Q)$ is in the limit-point case at infinity, i.e. the deficiency indices
$n_{\pm}(D^+_{\min}(Q))$ of $D^+_{\min}(Q)$ are the minimal possible:
    \begin{equation}\label{Intro1.3_Def_indic_Semiax}
n_{\pm}(D^+_{\min}(Q)) = \dim \ker (J_n\pm iI) =: \kappa_{\pm}(J_n) =: \kappa_{\pm},
%%\text{and}\quad  n_{\pm}(D_{\min}^{-}(Q)) = \dim \ker (J\mp iI),
   \end{equation}
(see \cite[Theorem 5.2]{LM03}).
%
%
%%for the case of continuous $Q$ and \cite[Theorem......]{LM03}  for  the
%%$L^1_{\loc}({\R},{\C}^{2n\times 2n})$-potential matrix $Q$) .

%%  In particular, if $n=2p$, and $\kappa_{+} = \kappa_{-} = p$, then  $n_{+}(D_{\min}^{+}(Q)) = n_{-}(D_{\min}^{+}(Q)) =p$,
%%the minimal symmetric operator $D_{\min}^{+}(Q)$ admits selfadjoint extensions in $L^2(\R_+,\C^{2n})$.
%% defined by appropriate boundary conditions  at zero given by special $p$ linearly %% independent  linear forms.
%%Going over to the semiaxis $\R_+$

We now  reformulate this statement  in terms of a    subspace $\th\subseteq\C^n$  and   define the operators  $D_{\th}^+(Q)$
% = D_{\th, \max}^+(Q)$
and $(D_{\th}^+)'(Q)$  to
be  the  restrictions of $\Dma^+(Q)$ to the domains
    \begin{gather*}
\dom  D_{\th}^+(Q) %:=  \dom  D_{\th, \max}^+(Q)
=\{y\in\dom \Dma^+(Q): y(0)\in \th \},
  \end{gather*}
and
  \begin{gather*}
\dom (D_{\th}^+)'(Q) =\{y\in \dom D_{\th}^+(Q): \  \supp y \subset [0,\infty)\ \text{and
is compact} \},
 \end{gather*}
respectively. Further   we   denote    by  $D_{\th, 0}^+(Q)$ % := D_{\th, \min}^+(Q)$
 the closure of
$(D_{\th}^+)'(Q)$.

With this notation  one easily gets that   the minimality condition
\eqref{Intro1.3_Def_indic_Semiax} is  equivalent to
     \begin{equation}\label{Intro1.3A_equal_domains}
D_{\th, 0}^+(Q) = D_{\th}^+(Q)  \quad \text{for any subspace} \quad \th\subseteq\C^n,
      \end{equation}
i.e.~the coincidence of the minimal and maximal operators.

%
%
%%It easily follows either from the Lagrange identity \eqref{3.2} (see e.g. the proof of
%%Theorem \ref{th3.3} (i)) or just from the obvious  relations
%
%%$$
%%\dim (D_{\th, \min}^+(Q)/D_{\min}^+(Q)) = \dim \th \quad \text{and}\quad \dim
%%(D_{\max}^+(Q)/ D_{\th, \max}^+(Q)) = \dim \th^\top. %%%%(\Bbb C^n\ominus \th)
%%$$
%
%

In the present paper we study  a general Dirac-type expression \eqref{Intro_1.1} with
arbitrary complex-valued $Q \in L^1_{\loc}({\mathcal I},{\C}^{n\times n})$  via the
abstract setting of dual pairs and $j$-selfadjoint operators. One of our main results
reads as follows.
%%can be formulated in the form of the following theorem.
%
%
  \begin{theorem}\label{th1.1_dom-max=dom-min}
For a general Dirac type expression  \eqref{Intro_1.1} on $\R$ with $Q \in
L^1_{\loc}({\R},{\C}^{n\times n})$ the following equality holds:
     \begin{equation}\label{Intro1.4_Main_result}
\Dmi(Q)=\Dma(Q).
     \end{equation}
\end{theorem}
In other words equality \eqref{Intro1.4_Main_result} means that the functions from $\dom
\Dma(Q)$ with compact support are dense in $\dom \Dma(Q)$ equipped with the  graph
norm.

By using  Theorem \ref{th1.1_dom-max=dom-min}  we obtain a  criterion for the minimal
operator $\Dmi (Q)$ in $L^2(\R,\C^{n})$  to be $j$-selfadjoint with an appropriate
conjugation $j$ (see Theorem \ref{th3.0.2}).

The latter result generalizes results  by Gesztesy, Cascaval, and Clark from
\cite[Theorem 4]{CG04} and \cite{CG06} (see also  \cite{CGHL04,CG02}).
Namely, assuming  that  $n = 2p$ they  %%%$J=i\diag(I_m, -I_m)$
proved  $j$-selfadjointness of the $j$-symmetric Dirac operator $D_{\min}(Q)$ for the particular case when
%%$\begin{pmatrix}
%%I_m & 0\\
%%& -I_m\end{pmatrix}$
%
%
   \be
J_n= \begin{pmatrix}
I_p & 0\\
0 & -I_p\end{pmatrix}, \quad  Q = -i
\begin{pmatrix}
0 & V\\
V^*&0\end{pmatrix} \quad \text{and} \quad  V=V^T \in L^1_{\loc}({\R},{\C}^{p\times p}).
 \ee
%
%

%% We note that this result  extends the result on coincidence of the maximal and minimal operators by Gesztesy, Cascaval, and Clark from
%%\cite[Theorem 4]{CG04} and \cite{CG06}  to  more general class of  potentials $Q$.
%%In our notation, their $Q$ is of the form % must satisfy $Q_{11}=Q_{22}=0$ and $Q_{12}= -iV$, $Q_{21}=-iV^*$
%%%\be
%%Q= -i
%%\begin{pmatrix}
%%0 & V\\
%%V^*&0\end{pmatrix},
%%\ee
%
%
Note in this connection that such   Dirac-type  operators naturally appear in the Lax
representation of the known $m\times m$ matrix valued AKNS systems (the focusing NLS
equation).

Moreover our above result, i.e. relation \eqref{Intro1.4_Main_result}, covers the
so-called vector NLS equation mentioned as an unsolved problem at the end of
\cite{CG04}, when
\be V= \begin{pmatrix}
           q_1&0&\ldots&0\\
q_2&0&\ldots&0\\
\vdots&\vdots&\ddots&\vdots\\
q_n&0&\ldots&0\\
\end{pmatrix}, \qquad q_j\in L^1_{\loc}(\R), \quad j\in \{1,\ldots,n\}.
\ee
It    naturally  appears  in the Lax representation of the vector NLS equation (cf.
\cite{Cher1996}),
$$
i\mathbf{q}_t + \frac{1}{2} \mathbf{q}_{xx} + \|\mathbf{q}\|^2 \mathbf{q} =0, \quad \mathbf{q} =
(q_1,\ldots, q_n)^{\top}, \quad \|\mathbf{q}\|^2 = \mathbf{q}^* \mathbf{q} = \sum |q_j|^2.
$$

Turning now  to the problem on the semiaxis $\R_+$ we first note that the counterpart of relation
\eqref{Intro1.3A_equal_domains}  remains valid in  the non-selfadjoint case, i.e.
$D_{\th, \min}^+ = D_{\th, \max}^+$ (see Theorem \ref{th3.3}). This result is applied
to prove the following theorem  %\ref{Intro_1.7_hom_equat}
on existence  of the  Weyl-type matrix solution to the
homogeneous equation
  \begin{equation}\label{Intro_1.7_hom_equat}
D(Q)y = J_n \frac{dy}{dx} + Q(x)y = zy, \quad x\in \R_+.
%%\Psi(z, \cdot)
  \end{equation}

  \begin{theorem}\label{Th1.2_Intro}
Let the minimal  Dirac operator $D_{\min}^+(Q)$  in $L^2(\R_+,\mathbb{C}^{2p})$  be  $j$-symmetric with respect to an appropriate  conjugation $j$.  Assume in addition that the field of regularity
$\widehat\rho\bigl(D^+_{\min}(Q)\bigr) \not = \emptyset$. Then
%%$\text{Def}\, D_{\min}^+(Q) = p$,
%
%
$$
\text{Def}\, D_{\min}^+(Q) := \dim\ker\bigl(D^+_{\max}(Q)-z\bigr)=p \quad \text{for any}
\quad z\in \widehat\rho\bigl(D^+_{\min}(Q)\bigr).
$$
Moreover, for any $z\in \widehat\rho\bigl(D^+_{\min}(Q)\bigr)$   there exists a $2p\times
p$-Weyl-type matrix solution $\Psi_+(z, \cdot)\in L^2(\R_+,\mathbb{C}^{2p\times p})$ of
equation \eqref{Intro_1.7_hom_equat} which can be chosen to be holomorphic in $z$.
%%$D^+_{\max}(Q)\Psi(z, \cdot)= z\Psi(z, \cdot)$.
      \end{theorem}
This theorem generalizes the corresponding result from  \cite[Theorem 5.4]{CGHL04} (see
also \cite{CG04} and \cite{CG06}) where it is proved for $p=1$ by another method.

Moreover, assuming that the imaginary part $\im Q$, of the potential matrix $Q$ is
bounded, i.e.   $\a_Q \le \im Q(x) \le \b_Q,$  $x\in\R_+$,   we show (see Proposition
\ref{pr3.11}) that
%%$n=2p$ and $\kappa_{\pm}(J) = p$,
%%under certain assumption on a potential matrix $Q$ that
%
%
\begin{gather}
\dim \gN_\l(\Dmi^+(Q^*)) =   \kappa_+, \qquad \dim \gN_{\overline\l}(\Dmi^+(Q)) =\kappa_-,\qquad \im\l > \b_Q, \nonumber   \\
\dim \gN_\l(\Dmi^+(Q^*)) = \kappa_-, \qquad \dim \gN_{\overline\l}(\Dmi^+(Q)) =
\kappa_+, \qquad \im\l < \a_Q. \label{Intro_1.7_def_indices}
 \end{gather}
 Here $\gN_\l(\Dmi^+(Q^*))= \ker((D^+_{\max}(Q) - \l)$. This result is a counterpart
 of the minimality conditions \eqref{Intro1.3_Def_indic_Semiax} in the non-selfadjoint case.

%%% $\C_{\a,+}=\{z\in\C: \im z> \a\}$ and $\C_{\a,-}=\{z\in\C: \im z < \a\}$

%%$A:= \Dmi(Q), \; B:=\Dmi(Q^*)$
%%%$\gN_\l(B)=\{f\in\dom B^*: B^*f=\l f\}$,
%%% $n_{+}(D_{\th, \min}^{+}(Q)) = n_{-}(D_{\th, \min}^{+}(Q)) =p$.

%%the minimal symmetric operator  $D_{\min}^{+}(Q)$ admits selfadjoint extensions in
%%$L^2(\R_+,\C^{2n})$ defined by appropriate boundary conditions  at zero.
%% We note that this result  extends the result on coincidence of the maximal and minimal operators
%%by Gesztesy, Cascaval, and Clark from
%%\cite[Theorem 4]{CG04} and \cite{CG06}  to  more general class of  potentials $Q$.
%%In our notation, their $Q$ is of the form % must satisfy $Q_{11}=Q_{22}=0$ and $Q_{12}= -iV$, $Q_{21}=-iV^*$

In the case $n=2p$ and $\kappa_{\pm} =p$ this result implies the existence of the $2p\times
p$ --Weyl-Titchmarsh-type matrix solution $\Psi(z, \cdot)\in L^2(\R_+,\mathbb{C}^{2p\times
p})$ of the equation $D^+_{\max}(Q)\Psi(z, \cdot)= z\Psi(z, \cdot)$ which complements
Theorem \ref{Th1.2_Intro}.

In turn, we apply this result to prove the uniqueness of the Weyl function for the Dirac
operator on the semiaxis (see Proposition \ref{pr3.15}). This result complements a recent result
from \cite{FKRS12} where it is proved for a skew-symmetric potential matrix $Q$ by
extending the classical Weyl limit point -- limit circle   procedure.

%
%%%In particular, we are able to determine the deficiency indices for many non-selfadjoint Dirac operators
%%This results extends  previously known results only in selfadjoint or $j$-selfadjoint situations.

The paper is organized as follows. Section \ref{sect1} introduces the abstract setting
that we will use: dual pairs of operators $A$ and $B$ and their proper extensions. Dual
pairs of operators and their extensions have been studied by many authors, see
e.g.~\cite{DM91b,GG91} for the symmetric case and
\cite{BGW09,BHMNW09,BMNW08,Lyantze,MM97,MM99,MM02} for the non-symmetric case. The main
idea of this paper consists of reducing the problem to the study of the properties of
the symmetric operator $S$ given by
\begin{equation}
S=
\begin{pmatrix}
0&A\\
B&0
\end{pmatrix}.
           \end{equation}
This allows us to relate the deficiency indices of $S$ to those of the dual pair and to
study $j$-symmetric operators $A$ and their extensions via $S$. In particular, we prove
here Proposition \ref{prop2.1}  containing an analog of the von Neumann formulas.

In turn, using this result we obtain the results of  Race \cite{Rac85}  and Zhikhar
\cite{Zhi59} by treating extensions of a  densely defined closed $j$-symmetric operator
$A$  as proper extensions of the dual pair $\{A, jAj\}$.

In Section \ref{Dirac} we apply  the abstract results to non-selfadjoint Dirac type operators
%
%
%%        \begin{equation}
%%D(Q)y:= J_{2n} \frac{dy}{dx} + Q(x)y,
%%        \end{equation}
%
%%\begin{equation}
%%J_{2n}=i
%%\begin{pmatrix}
%%I_n&0\\
%%0&-I_n
%%\end{pmatrix} \qquad \hbox{and}    \qquad
in  $L^2(\R,\C^{2n})$   to prove Theorem  \ref{th1.1_dom-max=dom-min},
as well as  $j$-selfadjointness of the  $j$-symmetric minimal Dirac-type operator $\Dmi (Q)$.
Moreover, we prove a  counterpart of Theorem  \ref{th1.1_dom-max=dom-min}
for $L^2(\R_+,\C^{2n})$ (see Theorem \ref{th3.3}).

 In Section 4.1 we  introduce  a boundary triple  for a dual pair of minimal Dirac operators $\{D_{\min}^+(Q),
D_{\min}^+(Q^*)\}$ and investigate the corresponding Weyl function. In Section 4.2 we
prove Theorem \ref{Th1.2_Intro} on existence  of the Weyl $2p\times p$-matrix solution
for $j$-symmetric operators.

In Section 4.3  we  investigate the Dirac-type operator on the half-line assuming a
potential matrix $Q$ with a bounded imaginary part. In particular, we prove the
equalities \eqref{Intro_1.7_def_indices} for such operators,  meaning the minimality of
deficiency indices in respective half-planes. Moreover,  we extract from this property
the uniqueness of the Weyl function for Dirac-type operators with $J_n$ satisfying
$\kappa_{\pm}=p$ (see Proposition \ref{pr3.15}). This result complements the results
from \cite{FKRS12} (see also \cite{SSR13}) and Proposition \ref{pr3.11}.
We also obtain certain additional properties of the Weyl function  for such Dirac-type
operators. In this connection we also mention the paper \cite{HMM13} where certain
properties of Dirac-type operators were investigated in the framework of boundary
triples and the  corresponding  Weyl functions.

\section{Dual pairs and $j$-symmetric operators} \label{sect1}

In this section we introduce several  abstract results for dual pairs of operators. These will be used to study  specific problems for Dirac operators in the rest of the paper.

\noindent{\bf 2.1. Dual pairs and symmetric operators.}

      \begin{definition}\label{def2.1}
(i) A pair of linear operators $A$ and $B$
in a Hilbert space $\gH$ forms a dual pair (or adjoint pair) $\{A,B\}$ if
               \begin{equation}\label{DP1}
(A f,g)=(f,Bg), \qquad   f\in\dom A,  \qquad g  \in\dom B.
                  \end{equation}
If $A$ and $B$ are closable densely defined operators, then \eqref{DP1} is equivalent to each of the following relations
                  \begin{equation}\label{DP2}
A\subseteq B^* \quad    \text{and} \quad  B\subseteq A^*.
                     \end{equation}

(ii) A densely defined operator ${\widetilde A}$ is called a proper extension of a dual pair $\{A,B\}$ if $A\subset{\widetilde A}\subset B^*$.
The set of all proper extensions of a dual pair $\{A,B\}$ is denoted by $\Ext\{A,B\}$.
       \end{definition}
%%%%%%%%%%%%%%%%%%%%%%%%%%%%%%%%%%%%%%%%%%%%%%%%%%%%%%%%%%%%%%%%%%%%

Our interest is in studying extensions in $\Ext\{A,B\}$. As a first step, we will investigate the relation between $\dom A$ and $\dom B^*$ and $\dom B$ and $\dom A^*$, respectively.

To any closable densely defined operator $A$ we associate
a Hilbert space $\gH_{+A}$ by  equipping  $\dom A^*$ with an inner product
              \begin{equation}\label{1}
(f,g)_{+A}=(f,g)+(A^* f, A^* g)\qquad   f,g \in \dom A^*.
                \end{equation}
By $\oplus_A$ we denote the orthogonal sum in  $\gH_{+A}$. For simplicity of notation, when no confusion is possible, we will omit specifying the operator and let $\gH_+:=\gH_{+A}$ with inner product $(f,g)_{+}:=(f,g)_{+A}$.

 \begin{remark} \label{rem:closed}   Let $\{A,B\}$ be a dual pair. Then,
clearly $\dom B$ is a (closed) subspace of  $\gH_{+A}$ if and only if $B$ is closed.
\end{remark}
%%%%%%%%%%%%%%%%%%%%%%%%%%%%%%%%%%%%%%%
    \begin{proposition}\label{prop2.1}
Let $\{A,B\}$ be a dual pair of densely defined closed operators in $\gH$. Then the following orthogonal decompositions hold
            \begin{equation}\label{2.2A}
\dom A^*=\dom B\oplus_A\ker(I+B^* A^*),
             \end{equation}
      \begin{equation}\label{2.3A}
\dom B^*=\dom A\oplus_B\ker(I+A^* B^*).
       \end{equation}
%%%%%%%%%%%%%%%%%%%%%%%%%%%%%
In particular, $A^*=B$ (or, equivalently, $A=B^*$) if and only if
        \begin{equation}\label{2.3B}
\ker(I+B^* A^*)=\{0\} \  \bigl(\hbox{or, equivalently, }\ker(I+A^* B^*)=\{0\}\bigr).
      \end{equation}
%%%%%%%%%%%%%%%%%%%%%%%
                \end{proposition}

\begin{remark} As in this proposition,
throughout the paper, many results will consist of two statements, the second one being the `adjoint' of the first. In these cases, we will usually omit the proof of the second statement, as it will follow analogously to the first.
\end{remark}

%%%%%%%%%%%%%%%%%%%%%%%%%%%%%%%%%%%%%%%%%%
         \begin{proof}
By Remark \ref{rem:closed}, $\dom B$ is a closed subspace of $\gH_+$. Let $g$ be orthogonal to $\dom B$ in $\gH_+$. Then for any $f\in\dom B$         %%%%%%%%%%%%%%%%%%%%%%%%%%%%%%%%%%%%%%%%%
          \begin{equation}\label{2.4A}
0=(f,g)_+=(A^* f,A^* g)+(f,g)=(B f,A^* g)+(f,g).%=\bigl(f,(I+B^* A^*)g\bigr),   \qquad    f\in\dom B.
             \end{equation}

Hence, $A^*g\in \dom B^*$ and $B^*A^*g=-g$, or $g\in\ker(I+B^* A^*)$. Conversely, if $g\in\ker(I+B^* A^*)$, then by \eqref{2.4A} we have $g\perp\dom B$. Thus \eqref{2.2A} is valid.
             \end{proof}
%%%%%%%%%%%%%%%%%%%%%%%%%%%%%%%%%%%%%%
   \begin{corollary}\label{cor2.1}
Let $\{A,B\}$ be a dual pair of densely defined operators in $\gH$. Define
\be
n(A^*,B) := \dim(\dom A^*/ \dom B)\quad \hbox{ and }\quad n(B^*,A):= \dim(\dom B^*/ \dom A).
\ee
 Then
            \begin{equation}\label{2.5A}
n(A^*,B) = n(B^*,A).                        \end{equation}
     \end{corollary}
%%%%%%%%%%%%%%%%%%%%%
         \begin{proof}
%%%%%%%%%%%%%%%%%%%%%%%%%%%%%%%%%%%%%%%%%%%%%%%%%
The simple equalities
$$ A^*(I+B^*A^*)=(I+A^*B^*)A^* \quad \hbox{ and } \quad B^*(I+A^*B^*)=(I+B^*A^*)B^*,$$
together with the fact that $A^*f\neq 0$ for $f\in\ker (I+B^*A^*)$ and $B^*g\neq 0$ for $g\in\ker (I+A^*B^*)$
imply that
           \begin{equation}\label{2.6A}
\dim\ker(I+B^*A^*)=\dim\ker(I+A^*B^*).
              \end{equation}
By combining this with \eqref{2.2A} and \eqref{2.3A}
we get the required result.
              \end{proof}
%%%%%%%%%%%%%%%%%%%%%%%%%%%%%%%%%%%%%%%%%

Formulas \eqref{2.2A} and \eqref{2.3A} may be considered as generalizations of the J. von Neumann formula known in the extension theory of  symmetric operators. In fact, the latter is easily implied by \eqref{2.2A} if $B=A$.                       %%%%%%%%%%%%%%%%%%%%%%%%%%%

For the next statement, for $\lambda$ in the field of regularity $\widehat \rho(A)$ of $A$, we use the common notation $\gN_{{\lambda}}(A) =
\ker(A^*-\lambda)$ for the defect subspace
 of the operator $A$,
 omitting the parameter $A$ when no confusion
can arise.
         \begin{corollary}\label{cor2.2}
 Let $A$ be a closed densely defined symmetric operator in $\gH$. Then
          \begin{equation}\label{7}
\dom A^*=\dom A\oplus_A\gN_i\oplus_A\gN_{-i}.
              \end{equation}
  \end{corollary}
           \begin{proof}
Note that, since $A$ is symmetric, the pair $\{A,A\}$ is a dual pair of operators in $\gH$. Now formula \eqref{2.2A} becomes
     \begin{equation}\label{8}
\dom A^*=\dom A\oplus_A\ker(I+A^{* 2})
=\dom A\oplus_A\ker(A^*+i I)\oplus_A\ker(A^*-i I).
      \end{equation}
      Here, the last equality follows from $u=-\frac{(A^*-i)u-(A^*+i)u}{2i}$ for any $u\in D(A^*)$.
                \end{proof}
%%%%%%%%%%%%%%%%%%%%%%%%%%%%%%%%%%%%%%%%
We now introduce a symmetric operator $S$ associated with the dual pair $\{A,B\}$ which allows us to extend many results from the symmetric case to our more general setting.

          \begin{proposition}\label{prop2.2}
Let $A$ and $B$ form a dual pair of closed densely defined operators in $\gH$.
Then

(i) the operator
      \begin{equation}\label{2.1}
S=
\begin{pmatrix}
0&A\\
B&0
\end{pmatrix}\quad\hbox{ with }\quad \dom (S)=\dom (B)\times \dom (A)
           \end{equation}
is a symmetric operator in $\gH^2=\gH\oplus\gH$ with equal deficiency indices,
      \begin{equation}\label{2.1B}
n_+(S)=n_-(S) = n(A^*,B) = n(B^*,A).
        \end{equation}
%%%%%%%%%%%%%%%%%%%%%%%%%%%%%%%%%%%%%%%

(ii) $S$ is selfadjoint if and only if
    \begin{equation}
A=B^*.
       \end{equation}
             \end{proposition}
%%%%%%%%%%%%%%%%%%%%%%%%%%%%%%%%%%%%%%%%%%%%%%%%%%%%%%%%%
\begin{proof}
(i)
Note that the property of $S$  being symmetric,
         \begin{equation}\label{2.1A}
S=
\begin{pmatrix}
0&A\\
B&0
\end{pmatrix}
\subset
\begin{pmatrix}
0&B^*\\
A^*&0
\end{pmatrix}
=S^*,
      \end{equation}
%%%%%%%%%%%%%%%%%%%%%%%%%%%%%%%
is equivalent to saying that the operators $A$ and  $B$ form a dual pair of operators.
%%a symmetry of $T, \  T\subset T^*$ follows from the definition of a dual %%pair.

It is easily seen that
  \begin{equation}
\ker(S^*\pm i)=\{f:\  f=\{f_1,f_2\}, \quad   f_1=\pm iB^* f_2,\quad  f_2\in\ker(I+A^* B^*)\}.
      \end{equation}
%%%%%%%%%%%%%%%%%%%%%%%%%%%%%%%%%%%%%%%%%
Note that $B^* f_2\not =0$ for any $f_2\in\ker(I+A^* B^*)$. Therefore $B^*$ maps $\cH_0=\ker(I+A^* B^*)$  isomorphically onto $B^* \cH_0$.

Hence
   \begin{equation}
n_{\pm}(S)=\dim\ker(S^*\mp i)=\dim\ker(I+A^* B^*)
     \end{equation}
and the statement follows from the proof of Corollary \ref{cor2.1}.

(ii) This statement is implied immediately by (i).
           \end{proof}
%%%%%%%%%%%%%%%%%%%%%%%%%%%%%%%%%%%%%%%%%%%%%%%%
      \begin{remark}
(i) The decompositions \eqref{2.2A} and   \eqref{2.3A} can also be derived
from Proposition \ref{prop2.2}.

(ii) In \cite{Mol98} operators of the form \eqref{2.1} are studied. An upper bound for $n_+$ and $n_-$ is given in \cite[Lemma 2.1]{Mol98}, but equality is not shown there.

(iii) Note that for any  ${\widetilde A}\in \Ext\{A,B\}$, the operator $T=
\begin{pmatrix}
0&{\widetilde A}\\
{\widetilde A}^*&0
\end{pmatrix}$ is a selfadjoint extension of $S$. Hence the equality $n_+(S)=n_-(S)$ is immediate.
          \end{remark}
%%%%%%%%%%%%%%%%%%%%%%%%%%%%%%%%%%%%%%%%%%%%%%%%%%%%%%

\noindent{\bf 2.2. Correct dual pairs and quasi-selfadjoint extensions.}

The following new definition will play a central role in the paper.
%%%%%%%%%%%%%%%%%%%%%%%%%%%%%%%%%%%
             \begin{definition}\label{def2.2}
Let  $\{A,B\}$ be a dual pair of closed densely defined operators in $\gH.$

(i) We will call an extension ${\widetilde A}\in \Ext\{A,B\}$  quasi-selfadjoint (or
quasi-hermitian) if
      \begin{equation}\label{2.80}
\dim(\dom{\widetilde A}/\dom A)=\dim(\dom{\widetilde A}^*/\dom B).
      \end{equation}

(ii) A  a dual pair  $\{A,B\}$ will be called a correct dual pair if it admits
a quasi-selfadjoint extension.
                 \end{definition}
%%%%%%%%%%%%%%%%%%%%%%%%%%%%%%%%%%%%%%%
The following proposition gives necessary conditions
for a dual pair $\{A,B\}$ to be a correct  dual pair.
In particular, it shows that not any dual pair is correct.
%%%%%%%%%%%%%%%%%%%%%%%%%%%%%%%%%%%
    \begin{proposition}\label{prop2.3}
Let $\{A,B\}$ be a dual pair of operators in $\gH$ and ${\widetilde A}\in \Ext\{A,B\}$.

(i) The following identities hold
      \begin{eqnarray}\label{2.81}
      n(A^*, B)&=&
\dim(\dom B^*/\dom{\widetilde A}) + \dim (\dom A^*/\dom{\widetilde A}^*) \\
&=&\dim(\dom {\widetilde A}/\dom{A}) + \dim (\dom {\widetilde A}^*/\dom{B}). \nonumber
              \end{eqnarray}
%%%%%%%%%%%%%%%%%%%%%%%

(ii) If ${\widetilde A}$ is a quasi-selfadjoint extension of $\{A,B\}$,
then $n(A^*,B)$ is even and
          \begin{eqnarray}\label{2.82}
          {n(A^*,B)}/{2}&=&
\dim(\dom B^*/\dom{\widetilde A})\ =\ \dim(\dom A^*/\dom{\widetilde A}^*)  \\
&=&\dim(\dom {\widetilde A}/\dom{A})\ =\ \dim(\dom {\widetilde A}^*/\dom{B}). \nonumber
        \end{eqnarray}
          \end{proposition}
%%%%%%%%%%%%%%%%%%%%%%%%%%%%%%%%%%%%%%%%%%%%%
      \begin{proof}
(i) Introducing the operator  $T=
\begin{pmatrix}
0&{\widetilde A}\\
{\widetilde A}^*&0
\end{pmatrix}$
we note that
        \begin{equation}\label{2.100}
S=
\begin{pmatrix}
0&A\\
B&0
\end{pmatrix}
\subset T=
\begin{pmatrix}
0&{\widetilde A}\\
{\widetilde A}^*&0
\end{pmatrix}
=T^*\subset
\begin{pmatrix}
0& B^* \\
A^*&0
\end{pmatrix}
=S^*,
     \end{equation}
%%%%%%%%%%%%%%%%%%%%%%%%
that is $T$ is a selfadjoint extension of $S$.
Set $n_l=\dim(\dom B^*/\dom{\widetilde A}) + \dim (\dom A^*/\dom{\widetilde A}^*)$ and  $n_r=\dim(\dom {\widetilde A}/\dom{A}) + \dim (\dom {\widetilde A}^*/\dom{B})$.
It follows from  \eqref{2.100} that
      \begin{equation}\label{2.101}
\dim(\dom T/\dom S)=n_r, \qquad \text{and}  \qquad
\dim(\dom S^*/\dom T)=n_l.
      \end{equation}
Since $T$ is a selfadjoint extension of $S$, then  J. von Neumann's formulas
yield
    \begin{equation}\label{2.102}
n_l = n_r = n_{\pm}(S).
   \end{equation}
Combining  \eqref{2.102} with  \eqref{2.1B} we get the required result.

(ii) Now let ${\widetilde A}$ be a quasi-selfadjoint extension of $\{A,B\}$, that is \eqref{2.80} holds. It is clear that
      \begin{equation}\label{2.103}
n(A^*,B)=\dim(\dom A^*/\dom{\widetilde A}^*)+\dim(\dom{\widetilde A}^*/\dom B)             \end{equation}
and
   \begin{equation}\label{2.104}
n(B^*,A)= \dim(\dom B^*/\dom{\widetilde A}) + \dim(\dom{\widetilde A}/\dom A).
   \end{equation}
%%%%%%%%%%%%%%%%%%%%%%%%%%%%%%%%%%%
Combining these formulas with  \eqref{2.80} and \eqref{2.5A}
we get
      \begin{equation}\label{2.105}
\dim(\dom A^*/\dom{\widetilde A}^*)=\dim(\dom B^*/\dom{\widetilde A})
         \end{equation}
It follows from \eqref{2.81}, \eqref{2.105} and \eqref{2.102} that
            \begin{equation}\label{2.106}
n_l=2\ \dim(\dom B^*/\dom{\widetilde A})=2\ \dim(\dom A^*/\dom{\widetilde A}^*)=n(A^*,B).
           \end{equation}
The remaining equalities in \eqref{2.82} are now implied by combining \eqref{2.103}, \eqref{2.104} and \eqref{2.106}.
      \end{proof}
%%%%%%%%%%%%%%%%%%%%%%%%%%%%%%%%%%%%%%%%%%%%%
Next we present a criterion for an extension ${\widetilde A}\in \Ext\{A,B\}$
 to be quasi-selfadjoint.
%%%%%%%%%%%%%%%%%%%%%%%%%%%%%%%%%%%%%%%%%%
        \begin{proposition}
Let $\{A,B\}$ be a dual pair in $\gH$. Suppose that  for some extension ${\widetilde A}\in Ext\{A,B\}$ the following condition holds
      \begin{equation}\label{2.70}
\dim(\dom B^*/\dom{\widetilde A})=\dim(\dom{\widetilde A}/\dom A).
        \end{equation}
Then ${\widetilde A}$ is a quasi-selfadjoint extension of $\{A,B\}$.
        \end{proposition}
%%%%%%%%%%%%%%%%%%%%%%%%%%%%%%%%%%%%%%
       \begin{proof}
It is clear from \eqref{2.100}-\eqref{2.102} that
        \begin{equation}\label{2.71}
\dim(\dom B^*/\dom{\widetilde A})+\dim(\dom A^*/\dom{\widetilde A}^*) = n_{\pm}(S)
        \end{equation}
and
           \begin{equation}\label{2.72}
\dim(\dom \widetilde A /\dom{A})+ \dim(\dom {\widetilde A}^*/\dom{B})=n_{\pm}(S)
           \end{equation}
%%%%%%%%%%%%%%%%%%%%%%%%%%%%%%%
Combining \eqref{2.71} and \eqref{2.72} with \eqref{2.70} we get
         \begin{equation}\label{2.73}
\dim(\dom A^*/\dom{\widetilde A}^*)=\dim(\dom{\widetilde A}^*/\dom B).
          \end{equation}
On the other hand, combining \eqref{2.70} with \eqref{2.104}  we obtain
            \begin{equation}\label{2.74}
2\ \dim(\dom{\widetilde A}/\dom A)=n(B^*,A)=n_{\pm}(S),
             \end{equation}
%%%%%%%%%%%%%%%%%%%%%%%%%%%%%%%%%%%%%%
and combining  \eqref{2.73} with \eqref{2.103} we have
                  \begin{equation}\label{2.75}
2\ \dim(\dom{\widetilde A}^*/ \dom B)=n(A^*,B)=n_{\pm}(S).
                  \end{equation}
Comparing \eqref{2.74} with \eqref{2.75} we get the required result.
          \end{proof}
%%%%%%%%%%%%%%%%%%%%%%%%%%%%%%%%%%%%%%%%%%%%%%%%%%%%%%%%%%%%%

\noindent{\bf 2.3. Defect numbers of  dual pairs and correct dual pairs.}

Here, we introduce  a concept of defect numbers of a dual pair
and establish their connection with deficiency indices of a symmetric
operator $S.$
We first recall a standard definition.
  \begin{definition}
(i)  The field of regularity $\widehat{\rho}(A)$ of a closed linear operator $A$ is the
set of $\lambda\in\C$ such that  for some  $\eps >0$
       \begin{equation}
\|A f-\lambda f\|\ge\eps\|f\|, \  \  f\in\dom A.
       \end{equation}
%
%
%%  \end{definition}
%

(ii)  For a dual pair $\{A,B\}$ we let $\widehat\rho \{A,B\}:=\{\l\in\C:
\l\in\widehat\rho (A)\;\; {\rm and} \;\; \overline \l\in\widehat\rho (B)\}$.
     \end{definition}

Clearly,
$\rho (\wt A)\subset \widehat \rho \{A,B\}$ for each proper extension $\wt A \in
\Ext\{A,B\}$.
We continue with a simple and well-known lemma.
%%%%%%%%%%%%%%%%%%%%%%%%%%%
         \begin{lemma}\label{lem2.3.1}
Let $\{A,B\}$ be a dual pair and ${\widetilde A}\in \Ext\{A,B\}$ a proper extension with
nonempty resolvent set $\rho({\widetilde A})$. Then for any
$\lambda_0\in\rho({\widetilde A})$ the following direct decompositions hold
            \begin{equation}\label{2.28A}
\dom{\widetilde A}=\dom A \dot+({\widetilde
A}-\lambda_0)^{-1}\gN_{{\overline\lambda}_0}(A),
           \end{equation}
%%%%%%%%%%%%%%
          \begin{equation}\label{2.29A}
 \dom{\widetilde A}^*=\dom B \dot+({\widetilde A}^*-{\overline\lambda}_0)^{-1}\gN_{\lambda_0}(B).
    \end{equation}
              \end{lemma}

\begin{proof}
We prove \eqref{2.28A}. Since $\lambda_0\in\rho({\widetilde A})$, we have
$\lambda_0\in\widehat{\rho}(A)$ and thus $\ran(A-\lambda_0)$ is closed. Therefore,
$\gH=\ran(A-\lambda_0)\oplus \gN_{{\overline\lambda}_0}(A)$. Applying $({\widetilde
A}-\lambda_0)^{-1}$ to this equality gives the desired result.
\end{proof}

%%%%%%%%%%%%%%%%%%%%%%%%%%%%%%%%%%%%%%%%%%%%%%

   \begin{corollary}\label{cor2.11}
If there exists an extension ${\widetilde A}\in \Ext\{A,B\}$ with $\rho({\widetilde
A})\not =\emptyset$, then
     \begin{equation}
n(A):=\dim\gN_{{\overline\lambda}}(A)=const,\qquad  \lambda\in\rho({\widetilde A}),
     \end{equation}
%%%%%%%%%%
         \begin{equation}
n(B):=\dim\gN_{\lambda}(B)=const,\qquad  {\overline\lambda}\in\rho({\widetilde A}^*).
          \end{equation}
                 \end{corollary}
%%%%%%%%%%%%%%%%%%%%%%%%%%%%%%%%%%%
     \begin{proof}
The proof is immediate from \eqref{2.28A} and \eqref{2.29A} since
    \begin{equation}\label{defn}
\dim\gN_{{\overline\lambda}_0}(A)=\dim(\dom{\widetilde A}/\dom A), \qquad\lambda_0\in\rho({\widetilde A}),
     \end{equation}
     and
%%%%%%%%%%%%%%%%%%%%%%%%%%%%%%%%%%%%%%
         \begin{equation}\label{defnm}
\dim\gN_{\lambda_0}(B)=\dim(\dom{\widetilde A}^*/\dom B), \qquad\lambda_0\in\rho({\widetilde A}).
              \end{equation}
%%%%%%%%%%%%%%%%%%%%%%%%%
\end{proof}
%%%%%%%%%%%%%%%%%%%%%%%%%%%%%%%%%%%%%%%
   \begin{definition}\label{def2.3}
The numbers $n(A)$ and $n(B)$ are called the defect numbers of a dual pair $\{A,B\}$.
       \end{definition}
%%%%%%%%%%%%%%%%%%%%%%%%%%%%%

We emphasize that $n(A)$ and $n(B)$ do not depend on the choice of ${\widetilde A}\in
\Ext\{A,B\}$ with $\rho({\widetilde A})\not =\emptyset$.
%%%%%%%%%%%%%%%%%%%%%%%%%%%%%%%%%%%%

Note that, by definition,  $\l_0\in {\widehat\rho}(A,B)$ if there exists $\eps > 0$ such
that
       \begin{equation}\label{2.42_reg-point-for_DP}
\|A f-\lambda_0 f\|\ge\eps\|f\|, \quad   f\in \dom A \quad\text{and}\quad \|Bg
 - \overline \lambda_0 g\|\ge\eps\|g\|, \quad  g \in \dom B.
           \end{equation}
We set $\mathbb{D}(\l_0; \varepsilon) := \{\l\in{\C}: |\l - \l_0| < \varepsilon\}.$ One
easily proves that  $\mathbb{D}(\l_0; \varepsilon) \subset {\widehat\rho}(A,B).$

The next result which can be found in \cite{MM97} and \cite[Proposition 2.10]{Mal01}
proves that there is a proper extension  preserving the  `gap in the spectrum'.
        \begin{proposition}\label{propMMM}
Let $\{A,B\}$ be a dual pair of densely defined operators in $\gH$. Suppose that
for some $\lambda_0\in\C$ we have
\begin{equation}
\|{Af-\lambda_0f}\|\geq\eps\|{f}\|,\ f\in D(A)\quad\hbox{ and }\quad \|Bg-\lambda_0g\|
\geq\eps\|g\|,\ g\in D(B).
  \end{equation}
 Then there exists  an extension
${\widetilde A}\in \Ext\{A,B\}$ preserving the gap, i.e.  such that
      \begin{equation}\label{*}
\mathbb{D}(\l_0; \eps)
%%U(\lambda_0,\eps):
=\{\lambda\in{\C}:|\lambda-\lambda_0|<\eps\}\subset\rho({\widetilde A})%, {\widetilde A}^*).
        \end{equation}
                and moreover,  the norm-preserving  estimate holds
     \begin{equation}\label{2.45_estimate_below}
\|{\widetilde A}f-\lambda_0f\|\geq\eps\|{f}\|,\quad  f\in D({\widetilde A}).
     \end{equation}
In particular, there exists a proper extension ${\widetilde A}\in \Ext\{A,B\}$ such that
$\lambda_0\in\rho({\widetilde A})$.
    \end{proposition}
%%%%%%%%%%%%%%%%%%%%%%%%%%%%%%%%%%%%%%%%%
%
  \begin{remark}\label{rem_gap_for_DP}
We emphasize  that estimate  \eqref{2.45_estimate_below} implies inclusion \eqref{*} but
not vice versa. These conditions are indeed equivalent  for a selfadjoint operator
${\widetilde A} = {\widetilde A}^*$.

Note also that in \cite{MM97} and \cite[Proposition 2.10]{Mal01} estimate
\eqref{2.45_estimate_below} was proved but was not stated  explicitly. In fact, this
estimate is useful in the sequel.
 \end{remark}
We have the following decompositions of the domains of $A^*$ and  $B^*$.
%%%%%%%%%%%%%%%%%%%%%%%%%%%%%%%%%%%%%%%%
        \begin{lemma} \label{lemMalMog}
 Let $\{A,B\}$ be a dual pair of densely defined closed operators in $\gH$. Suppose that $\l_0\in\widehat \rho
 \{A,B\}$,
 and let $\widetilde A$ be a proper extension of the dual pair with $\lambda_0\in\rho(\widetilde A)$.
 Then the following direct decompositions hold
       \begin{equation}\label{2.11}
\dom B^*=\dom A \dot+ ({\widetilde A}-\lambda_0)^{-1}\gN_{{\overline\lambda}_0}(A) \dot+
\gN_{\lambda_0}(B),
       \end{equation}
            \begin{equation}\label{2.12}
\dom A^* = \dom B \dot+ ({\widetilde A}^*-\overline{\lambda}_0)^{-1}\gN_{{\lambda}_0}(B)
\dot+ \gN_{{\overline\lambda}_0}(A).
             \end{equation}
       \end{lemma}

            \begin{proof}
            We prove \eqref{2.11}. Clearly, we have $\dom B^*\supseteq\dom A+({\widetilde A}-\lambda_0)^{-1}\gN_{{\overline\lambda}_0}(A)+\gN_{\lambda_0}(B)$. On the other hand, directness of the sum is easy to show and  for any $u\in \dom (B^*)$ we have
            \bea u &=&  u - ({\widetilde A}-\lambda_0)^{-1} (B^*-\lambda_0) u +  ({\widetilde A}-\lambda_0)^{-1} (B^*-\lambda_0) u
            \eea
            Since $\lambda_0\in{\widehat\rho}(A)$, we have that $\ran(A-\lambda_0)$ is closed and so $\gH=\ran(A-\lambda_0)\oplus \gN_{{\overline\lambda}_0}(A)$ and
            $$ (B^*-\lambda_0) u =  \left[ (A-\lambda_0)g+h\right]$$
            for some $g\in D(A)$ and $h\in \gN_{\overline{\lambda}_0}(A)$.
            Thus, \bea u
            &=& u - ({\widetilde A}-\lambda_0)^{-1} (B^*-\lambda_0) u +  ({\widetilde A}-\lambda_0)^{-1} \left[ (A-\lambda_0)g+h\right]\\
            &=& \left[u - ({\widetilde A}-\lambda_0)^{-1} (B^*-\lambda_0) u\right] +  g +({\widetilde A}-\lambda_0)^{-1} h.
            \eea
            Here, the first term on the right lies in $\gN_{\lambda_0}(B)$, $g$ is in $\dom(A)$ and  we have $({\widetilde A}-\lambda_0)^{-1} h\in (\widetilde{A}-\lambda_0)^{-1}\gN_{\overline{\lambda}_0}(A)$, proving the decomposition.
            \end{proof}
%%%%%%%%%%%%%%%%%%%%%%%%%%%%%%%%%
This result directly leads to some identities for deficiency indices.
          \begin{corollary}
Additionally to the conditions of Lemma \ref{lemMalMog}, let $S$ be an operator of the form  \eqref{2.1}. Then
        \begin{equation}\label{2.13}
n_{\pm}(S)=\dim\gN_{{\overline\lambda}_0}(A)+\dim\gN_{\lambda_0}(B).
           \end{equation}
      \end{corollary}
%%%%%%%%%%%%%%%%%%%%%%%%%%%%%%%%%%%%%%%
           \begin{proof}
It follows from \eqref{2.1B} and \eqref{2.11}  that
      \begin{equation}\label{2.14}
n(B^*,A)=\dim(\dom B^*/\dom A) =   \dim\gN_{{\overline\lambda}_0}(A) + \dim\gN_{{\lambda_0}}(B),
          \end{equation}
%%%%%%%%%%%%%%%%%%%%%%%%%%%%%%%%%%%
Comparing  \eqref{2.14}  with \eqref{2.1B} we get the required result.
                       \end{proof}
%%%%%%%%%%%%%%%%%%%%%%%%%%%%%%%%%%%%%%%%%%%%%%%%%%%%%%
  \begin{corollary} \label{cor:def}
Additionally to the conditions of Lemma \ref{lemMalMog} suppose that
        \begin{equation}\label{2.18}
\dim\gN_{{\overline\lambda}_0}(A)=\dim\gN_{\lambda_0}(B).
          \end{equation}
Then
       \begin{equation}\label{2.19}
n_{\pm}(S)=2\ \dim\gN_{{\overline\lambda}_0}(A) = 2\ \dim\gN_{\lambda_0}(B).                           \end{equation}
       \end{corollary}
%%%%%%%%%%%%%%%%%%%%%%%%%%%%%%%%%%%%%%%%%%%%%%%%%%%%%%%%%%%%%

\noindent{\bf 2.4. $j$-symmetric and $j$-selfadjoint operators.}

In this subsection, we study $j$-symmetric operators. Making use of the theory of dual
pairs we are able to show various results on deficiency indices for such operators. We
start with a definition. Recall that  $j$ is  a conjugation operator in $\gH$, if $j$ is
an anti-linear operator satisfying
        \begin{equation}
(ju, v)_{\gH}=(jv, u)_{\gH}, \quad   u,v\in\gH \qquad \text{and}\qquad j^2=I.
          \end{equation}
%%%%%%%%%%%%%%%%%%%%%%%%%%%%%
In particular,
     \begin{equation}
(ju, jv)_{\gH}=(v,u)_{\gH},  \qquad   u,v\in\gH.
         \end{equation}
%%%%%%%%%%%%%%%%%%%%%%%%
      \begin{definition} \label{def2.4}
(i) A densely defined linear operator $A$ in $\gH$ is called  $j$-symmetric if
    \begin{equation}
A\subseteq jA^*j \qquad  (\text{or equivalently, if}\  jAj\subseteq A^*).
        \end{equation}
%%%%%%%%%%%%%%%%%%%%%%%%%%%%%%

(ii) The operator $A$ is called $j$-selfadjoint if
    \begin{equation}
A=jA^*j \qquad (\text{or equivalently, if} \  jAj=A^*).
     \end{equation}
       \end{definition}
%%%%%%%%%%%%%%%%%%%%%%%%%%%%%

The following two results were originally  proved by Race \cite{Rac85}.
For the sake of completeness we present  proofs here
deducing them from Proposition \ref{prop2.1}.
%%%%%%%%%%%%%%%%%%%%%%%%%%%%%%%%%%%%%%%%%%%%
        \begin{proposition}  \label{propRace1}
 Let $A$ be a densely defined closed $j$-symmetric operator in $\gH$. Then the following decompositions hold
         \begin{equation}\label{9}
\dom A^*=\dom(j A j)\oplus_A\ker\bigl(I+(j A^*)^2\bigr)
             \end{equation}
%%%%%%%%%%%%%%%%%%
         \begin{equation}\label{2.10}
\dom(j A^* j)=\dom A\oplus_{j A j}\ker\bigl(I+(A^* j)^2\bigr)
             \end{equation}
In particular, a $j$-symmetric operator $A$ is $j$-selfadjoint if and only if
             \begin{equation}\label{2.11B}
\ker\bigl(I+(j A^*)^2\bigr)=\{0\}\qquad (\text{or equivalently, if}\ \ker(I+(A^* j)^2)=\{0\}\bigr).
            \end{equation}
                  \end{proposition}
%%%%%%%%%%%%%%%%%%%%
             \begin{proof}
Let $B=jAj$. Then $\{A,B\}$ is a dual pair of closed densely defined linear operators. Noting that $B^*=j A^* j$ and applying Proposition \ref{prop2.1}
we get
      \begin{equation}
\dom A^*=\dom(j A j)\oplus_A\ker(I+j A^* j A^*)\\
=\dom(j A j)\oplus_A\ker\bigl(I+(j A^*)^2\bigr).
         \end{equation}
%%%%%%%%%%%%%%%%%%%%%%%%
Formula \eqref{2.11B} is implied by \eqref{2.3B}.
          \end{proof}
%%%%%%%%%%%%%%%%%%%%%%%%%%%%%%%%%%%%%%%%%%%%%%%%%

%The following result has originally been proved by D.Race[].
%%%%%%%%%%%%%%%%%%%%%%%%%%%%%%%%%%%%%%%%
         \begin{proposition}\label{propRace2}
Let $A$ be a closed $j$-symmetric operator in $\gH$ and ${\widetilde A}$ any $j$-selfadjoint extension of $A$. Then
        \begin{equation}\label{2.108}
\dim\bigl(\dom(j A^* j)/\dom{\widetilde A}\bigr)=\dim\bigl(\dom\widetilde A/\dom A\bigr).
          \end{equation}
%%%%%%%%%%%%%%%%%%
                \end{proposition}
%%%%%%%%%%%%%%%%%%%%%%
    \begin{proof}
Setting $B=j A j$ we obtain a dual pair $\{A,B\}$. It is clear that
   \begin{equation}
S=
\begin{pmatrix}
0&A\\
j A j&0
\end{pmatrix}
\subset T=
\begin{pmatrix}
0&{\widetilde A}\\
{\widetilde A}^*&0
\end{pmatrix}
=T^*\subset
\begin{pmatrix}
0&j A^* j\\
A^*&0
\end{pmatrix}
=S^*.
     \end{equation}
Since ${\widetilde A}^*=j{\widetilde A}j$ and
   \begin{equation}\label{2.109}
j\dom A^*=\dom(j A^* j), \  j\dom{\widetilde A}=\dom(j{\widetilde A}j)=\dom{\widetilde A}^*, \ j\dom A=\dom(j A j),
    \end{equation}
we get
     \begin{equation}
\dim\bigl(\dom(j A^* j)/\dom{\widetilde A}\bigr)=\dim\bigl(\dom A^*/\dom{\widetilde A}^*\bigr).
    \end{equation}
It follows that
      \begin{equation}\label{2.110}
\dim(\dom S^*/\dom T)=2\ \dim\bigl(\dom(j A^* j)/\dom{\widetilde A}\bigr),
         \end{equation}
%%%%%%%%%%%%%%%%%%%%%%
               \begin{equation}\label{2.111}
\dim(\dom T/\dom S) = 2\ \dim\bigl(\dom{\widetilde A}/\dom{A}\bigr).
                      \end{equation}
%%%%%%%%%%%%%%%%%%%%
Since $T$ is a selfadjoint extension of $S$, then by J.von Neumann's formulas
         \begin{equation}\label{2.112}
\dim(\dom S^*/\dom T) =  \dim(\dom T/\dom S).
         \end{equation}
Combining \eqref{2.110}, \eqref{2.111} and \eqref{2.112} we get the required result.
               \end{proof}
%%%%%%%%%%%%%%%%%%%%%%%%%%%
Next we complement Proposition  \ref{propRace2} by the following simple statement.
%%%%%%%%%%%%%%%%%%%%%%%%%%%%%%%%%%%%%%%
     \begin{proposition}\label{propRacecomplement}
Let $A$ be a closed $j$-symmetric operator in $\gH$. Then any $j$-selfadjoint extension ${\widetilde A}$ of $A$ is a quasi-selfadjoint extension of the dual pair $\{A, jAj\}$.
        \end{proposition}
%%%%%%%%%%%%%%%%%%%%%%%
                 \begin{proof}
Since $\dom(j A j)=j\dom A$ and ${\widetilde A}^*=j{\widetilde A}j$,  taking \eqref{2.109} into account, we obtain  that
           \begin{eqnarray}
\dim\bigl(\dom{\widetilde A}^*/\dom(j A j)\bigr)&=&\dim\bigl(\dom(j{\widetilde A}j)/\dom(j A j)\bigr)   \\
&=& \dim(\dom{\widetilde A}/\dom A). \nonumber
             \end{eqnarray}
By Definition \ref{def2.2} this means that ${\widetilde A}$ is   quasi-selfadjoint.
\end{proof}
%%%%%%%%%%%%%%%%%%%%%%%%%%%%%%%%%%%%%%%%%%%%

              \begin{lemma}\label{lem2.2}
Let $A$ be a densely defined closed $j$-symmetric operator in $\gH$. Then:
\begin{enumerate}\def\labelenumi{\rm (\roman{enumi})}
\item  the operator
    \begin{equation}\label{2.65_DP_for_j-symmetric}
S=
\begin{pmatrix}
0&A\\
jAj&0
\end{pmatrix}
   \end{equation}
is a closed symmetric operator in $\gH^2=\gH\oplus\gH$ and $n_+(S)=n_-(S);$

\item $S$ is selfadjoint if and only if $A$ is $j$-selfadjoint.
\end{enumerate}
    \end{lemma}
%%%%%%%%%%%%%%%%%%%%%%%%%%%
      \begin{proof}
Setting $B:=jAj$, we note that by Definition \ref{def2.2}, we have $A\subset B^*=j A^*
j$ and $B\subset A^*$. Thus, $\{A,B\}$ forms a dual pair of densely defined closed
operators in $\gH$. It remains to apply Proposition  \ref{prop2.2} to obtain the desired
result.
           \end{proof}
%%%%%%%%%%%%%%%%%%%%%%%%%%%%%%%%%%%%%%%%%
%
   \begin{corollary}\label{cor2.25}
 Let $A$ be a $j$-symmetric operator in $\gH$, with a nonempty field of regularity ${\widehat\rho}(A)$
 and let $S$ be the operator of the form  \eqref{2.65_DP_for_j-symmetric}.
%
%%\begin{enumerate}\def\labelenumi{\rm (\roman{enumi})}
%%   \item
%%If $\l_0\in {\widehat\rho}(A)$, then $\mathbb{D}(\l_0; \varepsilon) := \{\l: |\l - \l_0| < \varepsilon\} \subset {\widehat\rho}(A)$
Then the following relations hold
     \begin{equation}\label{2.20}
2n(A) := 2\text{Def}\, (A)\  := 2\dim\gN_{\overline{\lambda}_0}(A)=n_{\pm}(S),\qquad \hbox{ for any
}\quad \lambda_0\in{\widehat\rho}(A).
      \end{equation}
%
%
%%\end{enumerate}
%
   \end{corollary}
%%%%%%%%%%%%%%%%%%%%%%%%%%%%%%%%%%%%%
             \begin{proof}
Let $\lambda_0\in{\widehat\rho}(A)$. Note that
${\overline\lambda}_0\in{\widehat\rho}(B)$, since for any $f\in\dom A$
    \begin{equation}\label{2.49}
\|Bj f-{\overline\lambda}_0 jf\|=\|j(A f-\lambda_0 f)\|=\|A f-\lambda_0 f\|.
     \end{equation}
We next show that equality \eqref{2.18} is satisfied for any
$\lambda_0\in{\widehat\rho}(A)$. Indeed, for any $f\in\ker(A^*-{\overline\lambda_0})$ we
have with $g=jf$,
        \begin{equation}
B^* g=j A^* j g= jA^*f= j\overline{\lambda}_0f=\lambda_0 jf =\lambda_0 g,
         \end{equation}
that is $g\in\ker(B^*-\lambda_0)$ and $j\ker(A^*-{\overline\lambda}_0)=\ker(B^*-\lambda_0)$.
Hence $\dim\gN_{\bar{\lambda}_0}(A) = \dim\gN_{\lambda_0}(B)$
Now \eqref{2.20} is implied by Corollary \ref{cor:def}.
                  \end{proof}
%%%%%%%%%%%%%%%%%%%%%%%%%%%%%%%%%%%%%%%%%%%%%%%%%%%%%%%%%%%%%%%%%%%%%%%%%
The next result is similar to Proposition \ref{propMMM} and shows that there is a gap preserving $j$-selfadjoint extension.
          \begin{proposition}\label{prop_gap_preserving}
Let $A$ be a $j$-symmetric operator in $\gH$. Suppose that
       \begin{equation}\label{2+}
\|A f-\lambda_0 f\|\ge\eps\|f\|, \qquad   f\in \dom A,
           \end{equation}
in particular,  $\mathbb{D}(\l_0; \varepsilon) \subset{\widehat\rho}(A)$. Then there
exists a $j$-selfadjoint extension ${\widetilde A}$ of $A$ such that
%%$\mathbb{D}(\l_0; \varepsilon) \subset \rho(\widetilde A)$ and
%%$\lambda_0\in\rho({\widetilde A})$ and moreover,
%
        \begin{equation}\label{3+}
\|{\widetilde A}f-\lambda_0 f\|\ge\eps\|f\|, \quad  f\in\dom{\widetilde A}.
         \end{equation}
%%%%%%%%%%%%%%%%%%%%%%%%
%
%%In other words,
In particular, the $j$-selfadjoint extension ${\widetilde A}$  preserves  the gap
$\mathbb{D}(\l_0; \varepsilon)$, i.e. $\mathbb{D}(\l_0; \varepsilon) \subset
\rho(\widetilde A)$.
    \end{proposition}
%%%%%%%%%%%%%%%%%%%%%%%%%%%%%%
       \begin{proof}
Setting $B:=j A j$ we note that $\{A,B\}$ is a dual pair of operators and due to \eqref{2.49} $\|B g-{\overline\lambda}_0 g\|\ge\eps\|g\|, \  g\in\dom B$. By Proposition \ref{propMMM}  there exists a proper extension ${\widetilde A}_1$ of $\{A,B\},$ obeying  \eqref{3+}.
We set ${\mathcal T}:=\eps({\widetilde A}_1-\lambda_0)^{-1}$ and  further
       \begin{equation}
{\mathcal S}:=({\mathcal T}+j {\mathcal T}^* j)/2=\eps[({\widetilde A}_1-\lambda_0)^{-1}+j({\widetilde A}^*_1-{\overline\lambda}_0)^{-1}j]/2
        \end{equation}
%%%%%%%%%%%%%%%%%%%%%%%%%%%%%
Clearly, ${\mathcal S}$ is contractive and
  \begin{equation}
{\mathcal S}^*=\eps[({\widetilde A}_1^*-\overline{\lambda_0})^{-1}+j({\widetilde A}_1-{\lambda}_0)^{-1}j]/2,
        \end{equation}
so $j{\mathcal S}^*j={\mathcal S}$ and ${\mathcal S}$ is $j$-selfadjoint. Moreover, as $\tA_1\in\Ext\{A,B\}$, we have $\tA_1\supset A$ and $\tA_1^*\supset jAj$. Thus $j\tA_1^*j\supset A$. Therefore, ${\mathcal S}$ is an extension of the non-densely defined contraction ${\mathcal T}_1=\eps(A-\lambda_0)^{-1}:\ran(A-\lambda_0)\to\gH$. Altogether we have
    \begin{equation}
{\mathcal T}_1\subset {\mathcal S}=j {\mathcal S}^* j, \qquad   \|{\mathcal S}\|\le 1.
     \end{equation}
Assume ${\mathcal S}f=0$ for some $f$ and let $g\in\dom A$. Then using that ${\mathcal S}$ extends ${\mathcal T}_1$ we obtain
\bea
0&=&({\mathcal S}f, j(A-\lambda_0)g)\ =\ (f, {\mathcal S}^*j(A-\lambda_0)g)\\% =\ (f, j{\mathcal S}j j(A-\lambda_0)g)\\
&=& (f, j{\mathcal S}(A-\lambda_0)g)\ =\  (f, \eps j g)\ = \ (g, \eps j f).
\eea
Density of $\dom A$ implies that $\eps j f=0$, so $f=0$ and ${\mathcal S}$ has trivial kernel.
As ${\mathcal S}$ is an injective contraction, it is invertible and we can set
${\widetilde A}=\eps {\mathcal S}^{-1}+\lambda_0$. Then
\bea
\tA^*&=& \eps ({\mathcal S}^{-1})^*+\overline{\lambda_0} \ =\ \eps (\mathcal S^*)^{-1} +\overline{\lambda_0}\\
&=& \eps j \mathcal S^{-1}j+ \overline{\lambda_0} \ =\ j (\eps \mathcal S^{-1}+ \lambda_0) j \ = \ j\widetilde{A} j
\eea
 and $\tA$ is $j$-selfadjoint.
Let $g\in\dom A$. Then
\be
g=\frac1\eps {\mathcal T}_1(A-\lambda_0)g = \frac1\eps {\mathcal S}(A-\lambda_0)g
\ee
and $g\in\dom {\mathcal S}^{-1}$ with $(\eps {\mathcal S}^{-1}+\lambda_0)g=Ag$, so $\tA$ extends $A$. It satisfies \eqref{3+}, as ${\mathcal S}$ is a contraction.
\end{proof}
\begin{corollary}[\cite{Zhi59}]\label{Neim_for-la_gener-n}
Let $A$ be $j$-symmetric operator in $\mathfrak H$ and $\widehat\rho(A) \not = \emptyset.$
Then:

(i) For any $\lambda_0\in\widehat\rho({A})$ there exists a $j$-selfadjoint extension
$\wt A$ of $A$ with $\lambda_0\in\rho(\wt{A}).$

(ii) For any $\lambda_0\in \rho(\wt{A})$ the following direct decompositions hold
       \begin{equation}\label{2.11_j-sym}
\dom jA^*j = \dom A \dot+({\widetilde A}-\lambda_0)^{-1}\gN_{{\overline\lambda}_0}(A)
\dot+ j\gN_{\overline\lambda_0}(A),
       \end{equation}
      \begin{equation}\label{2.12_j-sym}
\dom A^*=\dom jAj \dot+ ({\widetilde A}^* -
\overline{\lambda}_0)^{-1}j\gN_{{\overline\lambda}_0}(B) \dot+
\gN_{{\overline\lambda}_0}(A).
             \end{equation}
\end{corollary}
\begin{proof}
(i) This statement is immediate from Proposition \ref{prop_gap_preserving}.

(ii) These formulas are implied by formulas \eqref{2.11} and \eqref{2.12}, respectively,
with $B = jAj$ if one takes the relation $j\gN_{\overline\lambda_0}(A)=
\gN_{\overline\lambda_0}(B)$ into account.
  \end{proof}
This result implies one more justification of  the existence of the defect number
$\text{Def}(A)$ for any $j$-symmetric operator $A$ (see Corollary \ref{cor2.25}).
   \begin{corollary}\label{cor_def_number}
Let $A$ be a $j$-symmetric operator in $\gH$ with a nonempty field of regularity
${\widehat\rho}(A)$ and let $B=jAj$, so that $\{A,B\}$ is a dual pair of densely defined operators in $\gH$. Then $\wh\rho (A)=\wh \rho \{A,B\}$ and
     \begin{equation}\label{2.20A}
\text{Def}\, (A)\  := \dim\gN_{\ov\lambda}(A)=\dim\gN_\l(B)\qquad \hbox{ for any}\quad \l\in\wh\rho \{A,B\}.
      \end{equation}
              \end{corollary}
%%%%%%%%%%%%%%%%%%%%%%%%%%%%%%%%%%%%%
\begin{proof}
By Proposition   \ref{prop_gap_preserving},   for any $\lambda_0 \in \widehat\rho(A)$
there exists  a   $j$-selfadjoint extension  $\wt{A}$ satisfying  $\lambda_0\in
\rho(\wt{A})$. It follows from  \eqref{2.11_j-sym} that
   \begin{equation}\label{2.55}
\dim (\dom jA^*j/ \dom A)  =   2\dim \gN_{{\overline\lambda}_0}(A)\quad \text{for
any}\quad \lambda_0\in{\widehat\rho}(A).
  \end{equation}
Hence $\dim\gN_{{\lambda}_0}(A)$  does not depend on $\lambda_0\in{\widehat\rho}(A)$. This
proves \eqref{2.20A}.
\end{proof}
\begin{remark}
The existence of a $j$-selfadjoint extension of a $j$-symmetric operator has been
well-known for a long time. It was proved by Zhikhar \cite[Theorems 2,3]{Zhi59}  for the
first time using the Vishik method \cite{Vis52}, see e.g. \cite[Section 22]{Gla63}
and~\cite[Theorem III.5.8 \& Theorem III.5.9]{EE89}. However, the preservation of the
gap (in the sense of \eqref{3+}) seems to be new. It can be considered as a
generalization of well known  result   of  Krein  \cite{Kre47} which ensures  existence of a
selfadjoint extension $\wt A$ preserving a gap of a symmetric operator $A$ with a gap.

It also complements a result by Brasche et al in \cite[Theorem 3.1]{BNW93}, which states
that there need not be a selfadjoint extension of a symmetric operator that commutes
with $j$.

Note also that Lemma \ref{lemMalMog} and  Corollary \ref{Neim_for-la_gener-n} generalize
Vishik's results from \cite{Vis52}.
   \end{remark}

\section{ Non-selfadjoint Dirac type operators} \label{Dirac}
\subsection{Dirac type expressions and relative operators}
We now consider applications of the previous theory to Dirac type operators, both on the semiaxis and the whole line.

In the  following  $\gH$, $\cH$ are Hilbert spaces, $\B(\gH_1,\gH_2)$ is the set of
bounded operators defined on $\gH_1$ with values in $\gH_2$ and $\B(\gH):=\B(\gH,\gH)$.
An operator $T\in \B(\C^m)$ will often be identified with its matrix
$T=(T_{ij})_{i,j=1}^m$ in the standard basis of $\C^m$.

Moreover, $AC_{\loc}(\cI,\C^m)$ is the set of all
$\C^m$-valued vector-functions, defined on an interval $\cI\subset\R$ and absolutely
continuous on each compact subinterval $[\alpha, \beta]\subset \cI$ , $L^2(\cI,\C^m)$ is
the Hilbert space of all Borel measurable functions $f:\cI\to\C^m$ with
$$\int\limits_{\cI}|f(t)|^2\, dt < \infty,\quad \hbox{ where } |f(t)|^2=\sum\limits_{j=1}^m |f_j(t)|^2.$$
$L^1_{\rm loc}(\cI, \B (\C^m))$ is the set of all operator-functions  $F:\cI\to \B (\C^m)$ integrable on each compact interval $[a,b]\subset \cI$, $L^2(\cI, \B (\C^{m_1}, \C^{m_2} ))$ is the set of all operator-functions
$F:\cI\to \B (\C^{m_1}, \C^{m_2} ))$ such that $\int\limits_{\cI}|F(t)|^2\, dt < \infty$
or, equivalently, $F(\cd)h\in L^2(\cI, \C^{m_2})$ for all $h\in \C^{m_1}$. Moreover, for
$\alpha\in\R$ we denote by $\C_{\alpha,+}$ and  $\C_{\a,-}$ the  open half-planes
$\C_{\a,+}=\{z\in\C: \im z> \a\}$ and $\C_{\a,-}=\{z\in\C: \im z < \a\}$, respectively.
%of the complex plane $\C$.

Let $\cI=\langle a,b\rangle, \; -\infty\leq a<b \leq\infty,$ be an interval in $\R$ (each of the endpoints $a $ or $b$ may belong or not belong to $\cI$). We recall that we are considering the Dirac type differential expression
         \begin{equation}\label{3.1}
D(Q)y:= J_{n} \frac{dy}{dx} + Q(x)y
        \end{equation}
on $\cI$ with the operator $J_n \in\B (\C^n)$ such that  $J_n^*=J_n^{-1}=-J_n$ and the
operator (matrix) potential $Q(x)$. In the following we suppose (unless otherwise
stated) that $Q\in L^1_{\rm loc}(\cI, \B (\C^n))$. Expression \eqref{3.1} is called
formally selfadjoint if $Q(x)=Q^*(x)$ (a.e. on $\cI$).

Let $\gH:=L^2(\cI,{\C}^{n})$ and let $\Dma(Q)$ be the maximal operator in $\gH$ defined by
     \begin{gather*}
\dom \Dma (Q)= \{y\in\gH: y\in {AC_{\loc}}(\cI,\C^n), D(Q) y\in\gH\}\\
\Dma (Q) y=D(Q) y,\;\; y\in\dom\Dma (Q).
     \end{gather*}
Moreover, define the preminimal operator $\Dmi'(Q):=\Dma (Q)\up \dom \Dmi'(Q)$, where $\dom \Dmi'(Q)$ is the set of all $y\in \dom \Dma (Q)$  such that $\supp y $ is compact and $y(a)=0 \,(y(b)=0)$ if $ a\in\cI$ (resp. $ b\in\cI$). It follows by standard techniques (see e.g.~\cite{LS91}) that $\Dmi'(Q)$ is a densely defined closable operator in $\gH$. The minimal operator $\Dmi (Q)$ is then defined as the closure  of $\Dmi' (Q)$ and  satisfies $(\Dmi(Q))^*=\Dma(Q^*)$. Since obviously $\Dmi(Q)\subset \Dma(Q)=(\Dmi (Q^*))^*$, it follows that $\{\Dmi(Q),\Dmi
(Q^*)\}$ forms  a dual pair of closed densely defined operators in $\gH$.

\subsection{ Dirac type operators on  the whole line}
We now prove our first main result, Theorem \ref{th1.1_dom-max=dom-min}.
  %\begin{theorem}\label{th3.0}
%Let $Q\in L^1_{\rm loc}(\R, \C^{n\times n})$ and let $D(Q)$ be   Dirac type expression
%on $\R$ given by \eqref{3.1}.
%%% be defined on the whole line $\R$ and
%Then the minimal and maximal Dirac type operators associated with $D(Q)$ in
%$\gH:=L^2(\R,{\C}^{n})$ coincide, i.e.
%%%Let also $\Dmi(Q)$ and $\Dma (Q)$ be the corresponding e minimal and maximal operators
%%%in $\gH:=L^2(\R,{\C}^{n})$. Then
       %\begin{equation}\label{3.1.1}
%\Dmi(Q)=\Dma(Q).
     %\end{equation}
%\end{theorem}
\begin{proof}[Proof of Theorem 1.1]
Let $\wt J_{2n}\in \B (\C^{2n})$ and $\wt Q(x)(\in \B (\C^{2n}))$ be given by
       \begin{equation}\label{3.1.2}
\wt J_{2n}=\begin{pmatrix} 0 & J_n\\ J_n & 0 \end{pmatrix}, \qquad \wt Q(x)=\begin{pmatrix} 0 & Q(x)\\  Q^*(x) & 0 \end{pmatrix}
       \end{equation}
Since obviously $\wt J_{2n}^*=\wt J_{2n}^{-1}=-\wt J_{2n}$ and $\wt Q^*(x)=\wt Q(x)$,
the equality
     \begin{equation}\label{3.1.3}
D (\wt  Q)\wt y=\wt J_{2n} \frac{d \wt y}{dx}+\wt Q(x) \wt y
       \end{equation}
defines a formally  selfadjoint Dirac  type expression $D(\wt Q)\wt y$ on $\R$ and therefore according to \cite[Theorem 3.2]{LM03} $\Dmi (\wt Q)=(\Dmi (\wt Q))^*(=\Dma (\wt Q))$. On the other hand, one can easily verify that
     \begin{equation*}
\Dmi (\wt Q)= \begin{pmatrix} 0 & \Dmi(Q)\\  \Dmi(Q^*) & 0 \end{pmatrix}.
     \end{equation*}
Hence by Proposition \ref{prop2.2}, (ii) $\Dmi(Q)=(\Dmi (Q^*))^*=\Dma (Q)$, which proves \eqref{Intro1.4_Main_result}.
\end{proof}
In the following we denote by $j_n$ the complex conjugation in $\C^n$, i.e.
     \begin{equation*}
j_n h=\ov h:=\{\ov h_1, \ov h_2, \dots, \ov h_n\}, \qquad h=\{h_1,h_2, \dots, h_n \}\in\C^n
     \end{equation*}
Let $n=2p$,      $U=\begin{pmatrix} 0 & I_p\cr I_p & 0 \end{pmatrix}$ and let
$\cI=\langle a,b\rangle$ be an interval in $\R$. Then the equalities
     \begin{equation}\label{3.1.4}
\wt j_n:=U j_n(=j_n U) \quad \text{and}\quad (jf)(t):=\wt j_n f(t), \;\; t\in\cI, \qquad
f\in L^2(\cI,\C^n)
     \end{equation}
define conjugations $\wt j_n$ and $j$ in $\C^n$ and $L^2(\cI,\C^n)$ respectively.

In what follows  $C^\top$ denotes  the matrix  transpose to the matrix $C$.
\begin{proposition}\label{pr3.0.1}
Let $n=2p$, let $\cI=\langle a,b\rangle$ be an interval in $\R$ and let $D(Q)$ be   the  Dirac
type expression \eqref{3.1} on $\cI$ with $Q\in L^1_{\rm loc}(\cI, \B (\C^n))$ and
     \begin{equation}\label{3.1.5}
J_{n}=i \begin{pmatrix}I_p&0\\0&-I_p \end{pmatrix}:\C^p \oplus\C^p\to\C^p
\oplus\C^p,\quad Q(x)=
\begin{pmatrix}Q_{11}(x)&Q_{12}(x)\\Q_{21}(x)&Q_{22}(x)\end{pmatrix}:\C^p
\oplus\C^p\to\C^p \oplus\C^p.
     \end{equation}
Assume also that $j$ is a conjugation in $\gH=L^2(\cI,\C^n)$ given by \eqref{3.1.4}. If
     \begin{equation}\label{3.1.6}
Q^\top_{11}(x)=Q_{22}(x),\quad  Q^\top_{12}(x)=Q_{12}(x),\quad  Q^\top_{21}(x)=Q_{21}(x) \quad (\text{a.e. on }\;\;\cI),
     \end{equation}
then the operator $\Dmi (Q)$ is $j$-symmetric and
    \begin{equation}\label{3.1.7}
j \Dma (Q) j = \Dma (Q^*), \qquad j \Dmi (Q) j = \Dmi (Q^*).
    \end{equation}
%%and the operator $\Dmi (Q)$ is $j$-symmetric (in \eqref{3.1.6}
Here the operator $Q_{ij}(x)(\in \B(\C^p))$ is identified with its matrix in the
standard basis of $\C^p$.

Conversely, if $Q\in L^2_{\rm loc}(\cI, \B(\C^n))$ and  $\Dmi (Q)$ is $j$-symmetric,
then \eqref{3.1.6} holds.
  \end{proposition}
  \begin{proof}
Since $j_n=j_p\oplus j_p$, it follows that
     \begin{equation*}
\wt j_n Q(x) \wt j_n =j_n U Q(x)U j_n=
\begin{pmatrix} j_p Q_{22}(x) j_p &j_p Q_{21}(x) j_p \cr j_p Q_{12}(x) j_p &j_p Q_{11}(x) j_p \end{pmatrix}
=\begin{pmatrix} \ov{Q_{22}(x)}& \ov{Q_{21}(x)}\cr \ov{Q_{12}(x)}& \ov{Q_{11}(x)} \end{pmatrix}.
     \end{equation*}
Therefore \eqref{3.1.6} is equivalent to
    \begin{equation}\label{3.1.8}
\wt j_n Q(x) \wt j_n = Q^*(x)  \;\; (\text{\rm a.e. on } \cI).
    \end{equation}
Moreover, one easily  checks  that
    \begin{equation}\label{3.1.9}
\wt j_n J_n \wt j_n =J_n.
    \end{equation}

Let \eqref{3.1.6} be satisfied. Then \eqref{3.1.8} and \eqref{3.1.9} hold, which implies the first equality in \eqref{3.1.7}. By using this equality one gets
     \begin{equation*}
\Dmi(Q)=(\Dma (Q^*))^*=(j\Dma (Q) j)^*=j(\Dma (Q))^* j=j\Dmi (Q^*)j.
     \end{equation*}
This proves the second equality in \eqref{3.1.7}. Moreover, since $\Dmi(Q^*)\subset \Dma(Q^*)=(\Dmi (Q))^*$, it follows from the second equality in \eqref{3.1.7} that $\Dmi(Q) $ is $j$-symmetric.

Conversely, let $Q\in L^2_{\rm loc}(\cI, \B(\C^n))$ and  let the operator $\Dmi (Q)$ be $j$-symmetric. It is easily seen that  for each compact interval $[\a,\b]\subset \cI$ and for each $h\in\C^n$ there exists a function $y\in\dom  \Dmi' (Q)$ such that $y(x)=\wt j_n h, \; x\in [\a,\b]$. Clearly, a function $y_0=j y$ satisfies $j y_0\in \dom\Dmi (Q), \; y_0(x)=h, \;x\in [\a,\b],$ and for this function   %one has
     \begin{equation*}
(j \Dmi(Q) j y_0)(x)=\wt j_n Q(x)\wt j_n h, \qquad (\Dma(Q^*)y_0)(x)= Q^*(x) h  \;\;\; (\text {a.e. on}\;\; [\a,\b]).
     \end{equation*}
This and the inclusion $j\Dmi(Q)j \subset \Dma (Q^*)$ give \eqref{3.1.8}, which in turn yields \eqref{3.1.6}.
\end{proof}

Now we are in position  to state a  criterion for the Dirac type operator $\Dmi(Q)$ to
be $j$-selfadjoint on the line $\R$.
%%on the potential $Q(x)$ for $j$-selfadjointness of
%%the operator $\Dmi(Q)$ on the whole line is given by the following theorem.
%
%
  \begin{theorem}\label{th3.0.2}
Let $n=2p$,   $D(Q)$ be  the   Dirac type expression \eqref{3.1} on $\R$ with $J_n$ and $Q(\cd)\in L^1_{\rm loc}(\R, \B(\C^n))$ given by \eqref{3.1.5} and let  $j$ be the conjugation in $L^2(\R_+,\C^n)$ given by \eqref{3.1.4}. If  the  relations \eqref{3.1.6} are fulfilled, then the operator $\Dmi(Q)(=\Dma(Q))$ is $j$-selfadjoint. Conversely, if  $Q\in L^2_{\rm loc}(\R, \B(\C^n))$ and  $\Dmi (Q)$ is $j$-selfadjoint,
then \eqref{3.1.6} holds.
\end{theorem}
   \begin{proof}
By Proposition \ref{pr3.0.1}, conditions  \eqref{3.1.6}   imply relations \eqref{3.1.7}.
On the other hand,  Theorem \ref{th1.1_dom-max=dom-min} ensures the equality  $(\Dmi(Q))^*=\Dma
(Q^*)=\Dmi(Q^*)$. Combining this relation  with the second one  in \eqref{3.1.7}
 yields
$$
j \Dmi (Q) j = \Dmi (Q^*) = (\Dmi(Q))^*.
$$
This proves the result.
  \end{proof}
The following result  is immediate from Theorems \ref{th1.1_dom-max=dom-min} and \ref{th3.0.2}.
  \begin{corollary}\label{cor3.0.4}
In addition to the  assumptions of Theorem \ref{th3.0.2}  let
    \begin{equation}\label{3.1.10}
Q(x)=i\begin{pmatrix}0& -q(x)\\-q^*(x )&0\end{pmatrix},
    \end{equation}
where $q\in L_{\rm loc}^1(\R,\C^{p\times p})$ and $q(x)=q^\top (x)$ (a.e. on $\R$). Then
$\Dmi(Q)=\Dma (Q)$ and the operator $\Dmi (Q)$ is $j$-selfadjoint.
\end{corollary}

 \begin{remark}\label{rem3.0.5}
Corollary \ref{cor3.0.4} was proved by another method in \cite[Theorem 3.5]{CGHL04} in
the scalar case ($p=1$) and in \cite[Theorem 4]{CG04} for general $p$. It was also
generalized for a wider class of off-diagonal potential matrices $Q$ in  \cite{Kla10}. We
emphasize however that our method of reduction of the problem to the self-adjoint case is
apparently applied to such problems here for the first time.
      \end{remark}
\subsection { Dirac type operators on the half-line}
In this subsection we assume that the  Dirac type expression \eqref{3.1} is defined on the
half-line $\R_+=[0,\infty)$. The corresponding  minimal and maximal operators in
$\gH:= L^2(\R_+,\C^n)$ will be   denoted  by $\Dmi (Q)=\Dmi^+(Q)$ and
$\Dma(Q)=\Dma^+(Q)$.
    \begin{proposition}\label{pr3.1}
For  expression \eqref{3.1} on $\R_+$  the    (Lagrange) identity
       \begin{equation}\label{3.2}
(\Dma^+ (Q)y,z)_\gH - (y, \Dma^+ (Q^*)z)_\gH=-(J_n y(0), z(0))_{\C^n}
       \end{equation}
holds for every $y\in\dom \Dma^+(Q)$ and $z\in\dom \Dma^+(Q^*)$.
\end{proposition}
  \begin{proof}
Let $\wt J_{2n}$ and   $\wt Q(x)$ be the same as in \eqref{3.1.2} in the proof of Theorem \ref{th1.1_dom-max=dom-min}, let $ D(\wt Q)\wt y$ be the formally selfadjoint expression \eqref{3.1.3} and let $\wt \gH:=L^2(\R_+, \C^{2n})=\gH\oplus\gH$. Then according to \cite[Theorem 3.2]{LM03}  the following Lagrange identity holds:
         \begin{equation}\label{3.4}
(\Dma^+ (\wt Q) \wt y,\wt z)_{\wt\gH} - (\wt y, \Dma^+ (\wt Q) \wt z)_{\wt\gH}=-(\wt J_{2n} \wt  y(0),  \wt z(0)), \quad \wt y, \wt z \in\dom \Dma^+ (\wt Q).
           \end{equation}
Let $y\in\dom \Dma^+(Q)$,  $z\in\dom \Dma^+(Q^*)$ and let $\wt y= 0\oplus y (\in \gH\oplus \gH), \; \wt z =z\oplus 0(\in \gH\oplus \gH)$. Clearly, $\wt y,\wt z \in AC_{\loc}(\R_+,\C^{2n})\cap \wt \gH$ and
         \begin{equation*}
D(\wt Q)\wt y=\{J_n y'(x)+Q(x)y(x), 0\}, \qquad D(\wt Q) \wt z=\{0,J_n z'(x)+Q^*(x)y(x)\},\;\;\;x\in\R_+.
          \end{equation*}
Therefore $D(\wt Q)\wt y \in \wt \gH, \;  D(\wt Q) \wt z \in \wt \gH$ and, consequently, $\wt y, \wt z\in \dom \Dma^+  (\wt Q)$ and  $\Dma^+  (\wt Q) \wt y = \{\Dma^+ (Q)y, 0\}, \; \Dma^+  (\wt Q) \wt z = \{0, \Dma^+  ( Q^*) z\}$. Hence
            \begin{equation*}
(\Dma^+  (\wt Q) \wt y ,\wt z)_{\wt \gH}- (\wt y, \Dma^+  (\wt Q) \wt z )_{\wt\gH}=(\Dma^+ (Q)y,z)_\gH - (y, \Dma^+ (Q^*)z)_\gH.
             \end{equation*}
Moreover,
         \begin{equation*}
(\wt J_{2n} \wt y(0), \wt z(0))_{\wt\gH}=(\{J_n y(0), 0\}, \{z(0),  0\}) _{\wt\gH} = (J_ny(0), z(0))_\gH
         \end{equation*}
and \eqref{3.4} yields \eqref{3.2}.
\end{proof}
   \begin{lemma}\label{lem3.2}
For any $h \in \C^n$ there exists $y\in\dom \Dma^+ (Q)$ with compact support such that $y(0)=h$.
\end{lemma}
\begin{proof}
In the case of   the  formally selfadjoint expression \eqref{3.1}  the statement  is well known (see
e.g. \cite{LM03}). In the  general case it suffices to apply the known
statement  to the  formally selfadjoint expression \eqref{3.1.3}.
\end{proof}
For a subspace $\th\subset\C^n$ we let
         \begin{equation}\label{3.5}
\th^\times=\C^n\ominus J_n\th.
         \end{equation}
and define the operators $(D_{\th,0}^+)'(Q)$ %=(D_{\th}^+)'(Q)$
and $D_{\th}^+(Q)$ %=D_{\th,{\rm max}}^+(Q)$
 to be restrictions of $\Dma^+(Q)$ to the domains
%%$\dom D_{\th,0}'$ and $\dom D_{\th}$ given by
%
%
  \begin{gather*}
\dom (D_{\th,0}^+)'(Q)=\{y\in \dom \Dma^+(Q):y(0)\in\th\;\;{\rm and}\;\; \supp y \;\; \text{is compact} \}
 \end{gather*}
%%\exists b_y>0 \;\; {\text {\rm such that }\;\; y(t)=0, \, t>b_y}\}\\
and
    \begin{gather*}
 \dom D_{\th}^+(Q) = \{y\in\dom \Dma^+(Q): y(0)\in \th \},
  \end{gather*}
respectively. Denote also by  $\Dtmi^+(Q)$ %=D_{\th, {\rm min}}^+(Q)$
 the closure of $(D_{\th,0}^+)'(Q)$.

\begin{theorem}\label{th3.3}
 With the above notations the following holds:
  \begin{enumerate}\def\labelenumi{\rm (\roman{enumi})}
\item
For any  subspace $\th$ of $\C^n$  the  operators $\Dtmi^+ (Q)$ and $\Dtma^+ (Q)$ coincide,
that is
         \begin{equation}\label{3.6}
D_{\th,0}^+(Q)=D_\th^+(Q).
         \end{equation}
Moreover,
         \begin{equation}\label{3.6.1}
(D_\th^+(Q))^* = D_{\th^\times}^+(Q^*).
         \end{equation}
\item  The equality
         \begin{equation}\label{3.7}
\wt A=D_\th^+(Q)
         \end{equation}
establishes a bijective correspondence  between all subspaces $\th\subset \C^n$ and all extensions $\wt A\in \Ext \{\Dmi^+(Q),\Dmi^+ (Q^*)\}$. In particular, $\Dmi^+(Q)=D_{\th_0}^+(Q)$ with $\th_0=\{0\}$ and hence
         \begin{equation}\label{3.8}
 \dom\Dmi^+(Q)=\{y\in\dom\Dma^+(Q): y(0)=0\}.
         \end{equation}
\item  The dual pair $\{\Dmi^+(Q),\Dmi^+(Q^*)\}$ is correct if and only if $n=2p$. If this condition is satisfied, then the equality \eqref{3.7} gives a  bijective correspondence  between all   subspaces $\th\subset \C^{2p}$ with $\dim\th=p$ and all quasi-selfadjoint extensions $\wt A\in \Ext \{\Dmi^+(Q),\Dmi^+ (Q^*)\}$.
\end{enumerate}
\end{theorem}
\begin{proof}
{\rm (i)} We first  show that
         \begin{equation}\label{3.9}
(D_{\th,0}^+(Q))^*=(\Dtma^+(Q))^*=D_{\th^\times}^+(Q^*).
         \end{equation}
Since $(\Dmi^+)'(Q)\subset (\Dtmi^+)'(Q)\subset\Dtma^+ (Q)$ and $((\Dmi^+)'(Q))^*=(\Dmi^+(Q))^*=\Dma^+(Q^*)$, it follows that $((\Dtmi^+)'(Q))^*\subset \Dma^+(Q^*)$ and $(\Dtma^+(Q))^*\subset \Dma^+(Q^*)$. This and the Lagrange identity \eqref{3.2} yield
         \begin{gather*}
\dom ((\Dtmi^+)'(Q))^*= \{z\in\dom \Dma^+(Q^*):(J_ny(0), z(0))=0, \, y\in\dom (\Dtmi^+)'(Q)\},\\
\dom (\Dtma^+(Q))^*= \{z\in\dom \Dma^+(Q^*):(J_ny(0), z(0))=0, \, y\in\dom \Dtma^+ (Q)\}.
         \end{gather*}
 It follows  from  Lemma \ref{lem3.2} that
         \begin{gather*}
\{y(0):y\in\dom (\Dtmi^+)'(Q)\} = \{y(0):y\in\dom \Dtma^+(Q)\}=\th.
         \end{gather*}
Therefore by \eqref{3.5} $\dom ((\Dtmi^+)'(Q))^*=\dom (\Dtma^+(Q))^*=\dom D_{\th^\times}^+(Q^*)$, which together with the equality $(\Dtmi'(Q))^*= (\Dtmi(Q))^*$ yields \eqref{3.9}. The equalities \eqref{3.6} and \eqref{3.6.1} are immediate from \eqref{3.9}.

{\rm (ii)} The equalities $\Dmi^+ (Q)=D_{\{0\}}^+(Q)$ and \eqref{3.8} directly follow from the  definition of $\Dmi^+ (Q)$ and statement {\rm (i)}. Next, for each subspace $\th\subseteq\C^n$ the inclusion $D_\th(Q)\in \Ext \{\Dmi^+(Q),\Dmi^+ (Q^*)\}$ is obvious. Conversely, let $\wt A\in \Ext \{\Dmi^+(Q),\Dmi^+ (Q^*)\}$ and let
         \begin{equation}\label{3.10}
         \th:=\{y(0):y\in\dom \wt A\}.
         \end{equation}
Then for any $y\in \dom \Dt^+ (Q)$ there exists $u\in\dom \wt A$ such that $y(0)=u(0)$ and hence $v:=y-u$ satisfies $v(0)=0$. Therefore by \eqref{3.8} $v\in \dom\Dmi^+ (Q)\subset \dom \wt A$ and, consequently, $y=u+v\in \dom\wt A$. Thus $\dom \Dt^+ (Q)\subset \dom\wt A$. Since the
inverse inclusion $\dom\wt A\subset\dom \Dt^+ (Q)$ is obvious, it follows that $\dom\wt A=\dom \Dt^+ (Q)$ and hence \eqref{3.7} holds with $\th$ given by \eqref{3.10}.

{\rm (iii)} It follows from \eqref{3.8} that for each extension $\wt A=\Dt (Q)\in \Ext\{\Dmi^+(Q),\Dmi^+(Q^*)\}$ the equality
         \begin{equation}\label{3.11}
\dim (\dom \wt A/\dom\Dmi^+(Q))=\dim \th
         \end{equation}
holds. Moreover, by \eqref{3.6.1} one has $\dim (\dom \wt A^*/\dom\Dmi^+(Q^*))=\dim
\th^\times$. Combining these facts with the obvious identity  $\dim\th
+\dim\th^\times=n$ and statement (ii), one gets the required statement.
\end{proof}
\begin{corollary}\label{cor3.3.0}
With the notations above
   \begin{equation}\label{3.11a}
n((\Dmi^+(Q^*))^*,\Dmi^+(Q))(=\dim (\dom D_{\max}^+(Q)/\dom\Dmi^+(Q)))= n.
 \end{equation}
\end{corollary}
\begin{proof}
It follows from Lemma \ref{lem3.2} that $\Dma^+(Q)=D_\th(Q)$ with $\th=\C^n$. Combining
this relation with  \eqref{3.11} yields \eqref{3.11a}.
  \end{proof}
  \begin{remark}\label{rem3.3.1}
If $n=2p$ and
      \begin{equation}\label{3.11.0}
J_n=\begin{pmatrix} 0& -I_p \cr I_p &0 \end{pmatrix}:\C^p\oplus\C^p\to  \C^p\oplus\C^p,
        \end{equation}
then a subspace $\th \in \C^p\oplus\C^p$ is called a linear relation in $\C^p$ and the
linear relation
      \begin{equation*}
\th^*=\C^n\ominus J_n \th= \{\{h,h'\}\in  \C^p\oplus\C^p:(h',k)_{\C^p}- (h,k')_{\C^p}=0 \;\; \text{\rm  for all}\;\; \{k,k'\}\in \th \}
      \end{equation*}
is called the adjoint relation of $\th$ (see e.g. \cite{Cod73}). Clearly, in this case
$\th^\times=\th^*$. Moreover, if  the  Dirac type expression \eqref{3.1}, \eqref{3.11.0} is formally selfadjoint and $\th=\th^*$, then by Theorem \ref{th3.3} we have that $D_\th^+ (Q)(= D_{\th,0}^+ (Q))$ is
a selfadjoint extension of $\Dmi^+ (Q)$.
\end{remark}

\begin{definition}\label{def3.4}
A pair of operators (matrices) $C_1,C_2\in \B (\C^p)$ will be called admissible if $\rank (C_1,C_2)=p$.
\end{definition}
\begin{corollary}\label{cor3.5}
Assume that $n=2p$, so that the  functions $y\in\dom\Dma^+ (Q)$ and $z\in \dom \Dma^+ (Q^*)$ admit the representation
         \begin{equation}\label{3.11.1}
y(t)=\{y_1(t),y_2(t) \} (\in \C^p\oplus\C^p),\qquad  z(t)=\{z_1(t),z_2(t) \} (\in \C^p\oplus\C^p), \quad t\in\R_+.
         \end{equation}
Then for each admissible operator pair $C_1,C_2\in \B (\C^p)$ the equalities (the boundary conditions)
         \begin{equation}\label{3.11.2}
\dom \wt A =\{y\in\dom\Dma^+ (Q): C_1 y_1(0)+C_2 y_2(0)=0\}, \quad \wt A=\Dma^+ (Q) \upharpoonright \dom\wt A
         \end{equation}
define a quasi-selfadjoint extension $\wt A \in \Ext\{\Dmi^+(Q), \Dmi^+(Q^*)\}$ and conversely for each such   extension $\wt A$ there exists an admissible operator pair $C_1,C_2\in \B (\C^p)$ such that \eqref{3.11.2} holds.
\end{corollary}
\begin{proof}
Clearly, for each admissible pair  $C_1,C_2\in \B (\C^p)$  the equality
         \begin{equation}\label{3.11.3}
\th=\{\{h_1,h_2\}\in \C^p\oplus\C^p: C_1 h_1+ C_2 h_2=0\}
         \end{equation}
defines a subspace $\th\in\C^{2p}$ with  dimension $\dim\th=p $ and conversely for each such  subspace $\th$  there exists an admissible pair  $C_1,C_2\in \B (\C^p)$ such that \eqref{3.11.3} holds. This and Theorem \ref{th3.3}, (iii) yield the desired statement.
\end{proof}
\begin{corollary}\label{cor3.6}
Let $A:=\Dmi^+(Q)$ and $B:=\Dmi^+(Q^*)$. Assume that there exists an extension $\wt A\in \Ext \{A,B\}$ such that $\rho (\wt A)\neq\emptyset $ and $\dim\gN_{\overline \l_0}(A)=\dim\gN_{\l_0}(B)=:p$ for some $\l_0\in\rho (\wt A)$. Then:
\begin{enumerate}\def\labelenumi{\rm (\roman{enumi})}
\item $n=2p$ and
         \begin{equation}\label{3.12}
\dim\gN_{\overline \l}(A)=\dim\gN_{\l}(B)=p,\quad \l\in\rho (\wt A);
         \end{equation}

 \item $\wt A$ is a quasi-selfadjoint extension of the dual pair $\{A,B\}$ (hence this dual  pair is correct) and $\wt A=D_\th(Q)$ with a subspace $\th\in\C^{2p}$ of the dimension $\dim\th=p$.
\end{enumerate}
Conversely, let $n=2p$ and let $\wt A\in \Ext \{A,B\}$ be a quasi-selfadjoint extension with  $\rho (\wt A)\neq\emptyset $. Then \eqref{3.12} is valid.
\end{corollary}
\begin{proof}
{\rm (i)}
Combining of \eqref{3.11a} with  \eqref{2.14} gives $n=2p$. Equality \eqref{3.12} is immediate from Corollary \ref{cor2.11}.

{\rm (ii)} This statement follows from \eqref{defn}, \eqref{defnm} and Theorem \ref{th3.3}, (iii).

Assume now that $n=2p$ and $\wt A\in\Ext \{A,B\}$ is a quasi-selfadjoint extension with $\rho (\wt A)\neq \emptyset$. Then by Theorem \ref{th3.3}, (iii) and \eqref{3.11} we have $\dim (\dom \wt A/ \dom \Dmi^+(Q))=p$, which in view of \eqref{defn} and \eqref{defnm} yields \eqref{3.12}.
\end{proof}

   \section{Weyl solutions and  Weyl functions for  Dirac type operators on the half-line} \label{Dirac_2}

\subsection{Boundary triples for dual pairs and their Weyl functions}
We first recall the definitions of a boundary triple for a dual pair $\{A,B\}$ of closed densely defined  operators $A,B$ in $\gH$ and their respective $\gamma$-field and   Weyl function.
\begin{definition}\label{def3.7}$\,$\cite{Lyantze}
A  collection $\Pi=\bta$ consisting of Hilbert spaces $\cH_0$ and $\cH_1$ and linear mappings
        \begin{equation*}
\Gamma^B=\begin{pmatrix}\Gamma ^B_0\cr\Gamma ^B_1\end{pmatrix}: \dom B^*\to \cH_0\oplus \cH_1\ \ \text{and} \quad
\Gamma^A=\begin{pmatrix}\Gamma^A_0\cr\Gamma ^A_1\end{pmatrix}: \dom A^*\to \cH_1\oplus \cH_0
        \end{equation*}
is called a boundary triple for  $\{A,B\}$ if the mappings $\G^B$ and $\G^A$ are surjective and the following Green identity holds:
        \begin{equation}\label{3.14}
(B^*f,g)-(f,A^*g)=(\Gamma ^B_1 f,\Gamma ^A_0  g)- (\Gamma ^B_0
f,\Gamma ^A_1 g), \qquad   f\in \dom B^*,\ g\in \dom A^*.
          \end{equation}
\end{definition}
With each boundary triple $\Pi = \bta$ for $\{A,B\}$ one associates an extension
$A_0\in \Ext \{A,B\}$ given by $A_0=B^*~\up~\ker \G_0^B$. Moreover, if  $\rho (A_0)\neq
\emptyset$, one associates with $\Pi$  a   $\g$-field and the corresponding Weyl
function.
%%then it follows from Lemma \ref{lemMalMog} (in fact, from decomposition \eqref{2.11})
%%that for each $\l\in\rho (A_0) $ the mapping $\G_0^B\up \gN_\l(B) $ is a topological
%%isomorphism of $\gN_\l(B)$ onto $\cH_0$. This fact enables one to give the following definition.
%
%
   \begin{definition}\label{def3.8}$\,$\cite{MM02}
The operator-functions $M(\cd):\rho(A_0)\to \B (\cH_0,\cH_1)$ and $\g(\cd):\rho(A_0)\to \B (\cH_0,\gH)$ defined by
        \begin{equation}\label{3.15}
\G_1^B\up \gN_\l(B)=M(\l)(\G_0^B\up \gN_\l(B)) \;\;\; {\rm and} \;\;\;  \g(\l)=(\G_0^B\up \gN_\l(B))^{-1},\quad \l\in\rho (A_0)
        \end{equation}
are called the Weyl function and the $\g$-field of the boundary triple $\Pi=\bta$ for
$\{A,B\}$ respectively.
\end{definition}
Note that Lemma \ref{lemMalMog} ensures  the $\g$-field $\g(\cdot)$ and Weyl function
are well defined. Indeed, it follows from the  decomposition \eqref{2.11} that for each
$\l\in\rho (A_0) $ the mapping $\G_0^B\up \gN_\l(B)$ is a topological isomorphism of
$\gN_\l(B)$ onto $\cH_0$, hence $\g(\cdot)$ is well defined on  $\rho(A_0)$. Moreover,
it is shown in \cite{MM02} that $M(\cd)$ and $\g(\cd)$  are holomorphic in $\rho(A_0)$.
   \begin{remark}\label{rem3.8.1}
Let $S$ be a  closed densely defined symmetric operator in $\gH$. Then a collection $\Pi=\{\cH,\G_0,\G_1\}$ consisting of a Hilbert space $\cH$ and linear operators $\G_j:\dom S^*\to \cH, \; j\in\{0,1\},$ is called a boundary triple for $S^*$ if the operator  $\G=(\G_0,\G_1)^\top$ is surjective  and the abstract Green identity
        \begin{equation}\label{3.15.0}
(S^*f,g)-(f,S^* g)=(\G_1 f, \G_0 g)-(\G_0 f, \G_1 g), \quad f,g\in\dom S^*
        \end{equation}
is valid (see e.g. \cite{GG91}). In this connection note that
Definition \ref{def3.7} of a boundary triple $\{\cH_0\oplus\cH_1,\G^B,\G^A\}$ for a dual pair $\{A,B\}$ generalizes the above   definition of a boundary triple $\{\cH,\G_0,\G_1\}$ for  $S^*$  and coincides with it in the case $A=B=:S$ if additionally $\cH_1=\cH_0=:\cH$ and $\G^B=\G^A=:\G$. Observe also that in this case $S_0=S^*\up\ker\G_0$ is a selfadjoint extension of $S$ and  the Weyl function $M(\cd)$ of the triple $\Pi$ in the sense of Definition \ref{def3.8} turns into the Weyl function in the sense of \cite{DM91a}.
\end{remark}
\begin{lemma}\label{lem3.8.2}
Let $S$ be a closed densely defined symmetric operator in $\gH$ with the deficiency indices $n_\pm (S)=\dim \gN_\l(S), \; \l\in\C_\pm,$  let $T\in \B(\gH)$ be a selfadjoint operator, let $\a=\inf\limits_{f\in\gH} \frac {(Tf,f)}{||f||^2}$ and  $\b=\sup\limits_{f\in\gH}  \frac {(Tf,f)}{||f||^2}$ be   the  lower lower and the  upper bounds of $T$ respectively and let
        \begin{equation*}
        A=S+i T, \qquad B=S-iT.
        \end{equation*}
Then:
\begin{enumerate}\def\labelenumi{\rm (\roman{enumi})}
\item We have
 $\dom A=\dom B=\cD$, where $\cD:=\dom S$. Moreover,
        \begin{equation*}
\a ||f||^2\leq \im (Af,f)\leq  \b ||f||^2, \qquad - \b ||f||^2\leq\im (Bf,f)\leq -\a ||f||^2, \quad f\in\cD
         \end{equation*}
and the operators $A$ and $B$ form a dual pair \{A,B\}.
\item The inclusion
$(\C_{\a,-}\cup \C_{\b,+})\subset \wh \rho \{A,B\}$ holds and
        \begin{gather}
\dim \gN_\l(B)=n_+(S), \qquad \dim \gN_{\overline\l}(A) =n_-(S),\quad \l\in\C_{\b,+}\label{3.15.1},\\
\dim \gN_\l(B)=n_-(S), \qquad \dim \gN_{\overline\l}(A) =n_+(S),\quad \l\in\C_{\a,-}\label{3.15.2}.
\end{gather}
\item
The following equalities hold
        \begin{gather}\label{3.15.3}
\dom A^*=\dom B^*=\cD_*, \qquad A^*=S^*-iT, \quad B^*= S^*+iT,
         \end{gather}
where $\cD_*:=\dom S^*$.
\end{enumerate}
\end{lemma}
\begin{proof}
Statements (i) and (iii) are obvious. Statement (ii) is immediate from the known results
of the perturbation theory for linear operators (see e.g. \cite[Theorem 3.7.1]{BirSol80}) %%\cite{Ka}.
\end{proof}
\begin{theorem}\label{th3.8.3}
 With  the assumptions of Lemma \ref{lem3.8.2}  let   also  $n_+(S)=n_-(S)$, $\cD_*=\dom S^*$ and let $\wh\Pi =\{\cH,\G_0,\G_1\}$ be a boundary triple for $S^*$ (see Remark \ref{rem3.8.1}). Then:
\begin{enumerate}\def\labelenumi{\rm (\roman{enumi})}
\item
the collection $\Pi=\{\cH\oplus\cH, \G^B,\G^A\}$ with operators $\G^B=(\G_0^B,\G_1^B)^\top:\cD_*\to \cH\oplus\cH$ and $\G^A=(\G_0^A,\G_1^A)^\top:\cD_*\to \cH\oplus\cH$ defined by
        \begin{gather}\label{3.15.4}
\G_0^B=  \G_0^A=\G_0, \qquad   \G_1^B=  \G_1^A=\G_1
        \end{gather}
is a boundary triple for the dual pair $\{A,B\}$. Moreover, the  extension $A_0(=B^*\up\ker \G_0^B)\in\Ext \{A,B\}$ satisfies
        \begin{gather}\label{3.15.5}
\a  ||f||^2\leq  \im (A_0f,f)\leq \b  ||f||^2, \;\; f\in\dom A_0, \;\; \; {\rm and}\;\;\; (\C_{\a,-}\cup \C_{\b,+})\subset \rho (A_0);
       \end{gather}
\item
the Weyl function $M(\cd)$ of the triple $\Pi$ satisfies the identities
       \begin{gather}
M(\l)-M^*(\mu) =\g^*(\mu)[(\l-\overline \mu)I_\gH-2i T] \g(\l), \quad \l,\mu \in\rho (A_0)\label{3.15.6}\\
\im M(\l)=\g^*(\l)(\im\l\cdot I_\gH-T)\g(\l), \quad \l\in\rho (A_0),\label{3.15.7}
\end{gather}
where $\g(\l)$ is the $\g$-field of  $\Pi$. Moreover,
\begin{gather}\label{3.15.8}
\im M(\l)\geq \delta_\l I, \quad \l\in\C_{\b,+};  \qquad  \im M(\l)\leq \delta_\l' I, \quad \l\in\C_{\a,-}
\end{gather}
with some $\delta_\l>0$ and $\delta_\l'<0$ (depending on $\l$).
\end{enumerate}
\end{theorem}
\begin{proof}
(i) Since by \eqref{3.15.4} we have $\G^B=\G^A=\G$ and the mapping $\G$ is surjective,  so are the mappings $\G^B$ and $\G^A$. Moreover, using \eqref{3.15.3}, \eqref{3.15.4} and the identity \eqref{3.15.0} (for the triple $\wh \Pi$) one gets
       \begin{gather*}
(B^*f,g) - (f,A^*g)= (S^*f,g)+i(Tf,g)-(f,S^*g)-i(f,Tg)=(S^*f,g)- (f,S^*g)\\
=(\G_1 f, \G_0 g)- (\G_0 f, \G_1 g)=(\G_1^B f, \G_0^A g)- (\G_0^B f, \G_1^A g), \quad f,g\in\cD_*.
        \end{gather*}
Hence $\G^B$ and $\G^A$ satisfy \eqref{3.14} and, consequently,  $\Pi=\{\cH\oplus\cH, \G^B,\G^A\}$ is a boundary triple for $\{A,B\}$.

Next assume that $S_0=S^*\up \ker \G_0$ is a selfadjoint extension of $S$ corresponding to the triple $\wh \Pi$ for $S^*$. Then
       \begin{gather*}
A_0=B^*\up\ker\G_0^B=(S^*+iT)\up\ker\G_0=S_0+iT
       \end{gather*}
and Lemma \ref{lem3.8.2} with %taking the equalities
 $n_\pm(S_0)=0$  %into account
 yields \eqref{3.15.5}.

(ii) Since $\gN_\l(B)=\{f\in\dom B^*: B^*f=\l f\}$, it follows from \eqref{3.15.3} that
         \begin{gather}\label{3.15.11}
\gN_\l(B)=\{f\in\cD_*: S^*f=\l f-iTf\}, \quad \l\in\C.
        \end{gather}
Assume that $h\in\cH, \; \l,\mu\in \rho (A_0)$ and
       \begin{gather*}
f_\l=\g(\l)h\in \gN_\l(B), \qquad g_\mu=\g(\mu)h\in\gN_\mu(B).
        \end{gather*}
Then by \eqref{3.15.11}
       \begin{gather*}
S^* f_\l=\l f_\l-i T f_\l=(\l I-iT)\g(\l)h, \qquad S^* g_\mu=\mu g_\mu-i T g_\mu=(\mu I-iT)\g(\mu)h
       \end{gather*}
and hence
     \begin{gather*}
(S^* f_\l,g_\mu)- (f_\l,S^* g_\mu)=((\l I-iT )\g(\l)h,\g(\mu)h ) -(\g(\l)h, (\mu I-iT)\g(\mu)h)=\\
=(\g^*(\mu)[(\l I-iT)-(\overline \mu I +i T)]\g(\l)h,h)  =(\g^*(\mu)[(\l-\overline \mu)I-2i T]\g(\l)h,h)
            \end{gather*}
On the other hand, by the  definition of $\g(\l)$ and $M(\l)$ one has
        \begin{gather*}
\G_0f_\l=\G_0^B f_\l= h, \quad \G_1f_\l=\G_1^B f_\l= M(\l)h, \quad \G_0 g_\mu=\G_0^B g_\mu= h, \quad \G_1 g_\mu=\G_1^B g_\mu= M(\mu)h
        \end{gather*}
and, consequently,
        \begin{gather*}
(\G_1f_\l,\G_0 g_\mu )-(\G_0f_\l, \G_1 g_\mu)  = (M(\l)h,h)-(h,M(\mu)h)=((M(\l)-M^*(\mu))h,h).
        \end{gather*}
Now an  application of the  identity \eqref{3.15.0} for the triple $\wh \Pi$ to $f_\l$ and $g_\mu$ yields
          \begin{gather*}
(\g^*(\mu)[(\l-\overline \mu)I-2i T]\g(\l)h,h)= ((M(\l)-M^*(\mu))h,h), \quad h\in \cH, \;\; \l,\mu\in \rho (A_0).
          \end{gather*}
This implies identity \eqref{3.15.6}, which in turn gives \eqref{3.15.7}. Finally, \eqref{3.15.8} directly follows from \eqref{3.15.7}.
\end{proof}

\subsection{The Weyl solution and the Weyl function for Dirac operators}
In this subsection  we assume that the Dirac type expression \eqref{3.1} is defined on the
half-line $\R_+$. First we prove the existence of the Weyl solution on $\R_+$.
\begin{theorem}\label{th3.9}
Suppose that $n=2p$ and the operator $J_n$ in \eqref{3.1} is of the form \eqref{3.11.0}.
Let $A:=\Dmi^+(Q)$ and $B:=\Dmi^+(Q^*)$, so that $A^*=\Dma^+(Q^*)$ and $B^*=\Dma^+ (Q)$.
Moreover, let $C_1,C_2\in \B (\C^p)$ be an admissible operator pair and let $\wt A$ be
the corresponding  quasi-selfadjoint extension of the form \eqref{3.11.2}. If $\rho (\wt
A)\not = \emptyset$, then the following hold:
%%such that the corresponding  quasi-selfadjoint extension $\wt A$ of the form
%%\eqref{3.11.2} has the nonempty resolvent set $\rho (\wt A)$. Then:
%
%
  \begin{enumerate}\def\labelenumi{\rm (\roman{enumi})}
\item There exist a pair of operators
        \begin{equation}\label{3.17}
X=\begin{pmatrix}  C_1 & C_2 \cr C_3&C_4\end{pmatrix}:\C^p\oplus\C^p\to \C^p\oplus\C^p, \qquad Y=\begin{pmatrix}  C_1' & C_2' \cr C_3'&C_4'\end{pmatrix}:\C^p\oplus\C^p\to \C^p\oplus\C^p
        \end{equation}
such that $Y^* J_n X=J_n$ and for each such a pair  the equalities
        \begin{gather}
\G_0^B y= C_1y_1(0) + C_2y_2(0), \qquad   \G_1^B y= C_3y_1(0) + C_4y_2(0), \quad y\in\dom B^*,\label{3.18}\\
\G_0^A y= C_1'z_1(0) + C_2'z_2(0), \qquad   \G_1^A y= C_3'z_1(0) + C_4'z_2(0), \quad z\in\dom A^*,\label{3.19}\\
\G^B=(\G_0^B,\G_1^B)^\top, \qquad \G^A=(\G_0^A,\G_1^A)^\top, \label{3.20}
        \end{gather}
define a boundary triple $\Pi=\{\C^p\oplus\C^p,\G^B,\G^A\}$ for $\{A,B\}$ (in \eqref{3.18} and \eqref{3.19} $y_j$ and $z_j$ are taken from \eqref{3.11.1}). Moreover, for this triple $A_0(=B^*\up\ker\G_0^B)=\wt A$.

\item For each $\l\in \rho(\wt A)$
the homogeneous equation
        \begin{gather}\label{3.22}
D(Q)y=\l y
        \end{gather}
has a unique $L^2$-operator-valued  solution (the Weyl solution)
        \begin{gather}\label{3.21}
v(\cd,\l)=(v_1(\cd,\l), v_2(\cd,\l))^\top \in L^2(\R_+,\B (\C^p,\C^{2p}))
%%%\C^p\to \C^p\oplus\C^p, \quad t\in\R_+,
        \end{gather}
satisfying
%%%$v(\cd,\l)\in L^2(\R_+,\B (\C^p,\C^n))$ and
%
        \begin{gather}\label{3.23}
C_1 v_1(0,\l)+ C_2 v_2(0,\l)=I_p, \quad \l\in \rho(\wt A).
        \end{gather}
Moreover, the solution $v(\cd,\l) $ is holomorphic in $\l\in \rho(\wt A)$.

\item  The Weyl function $M(\cd)$ of the triple $\Pi$ is
        \begin{gather}\label{3.23.1}
M(\l)=C_3 v_1(0,\l)+ C_4 v_2(0,\l), \quad \l\in \rho(\wt A).
        \end{gather}
If $\ff(\cd,\l)$ and $\psi (\cd,\l)$ are $\B (\C^p, \C^p\oplus\C^p)$-valued solutions of
\eqref{3.22} satisfying
        \begin{gather}\label{3.24}
X\ff(0,\l)= \begin{pmatrix}  0\cr I_p \end{pmatrix}, \qquad X\psi(0,\l)= \begin{pmatrix}  I_p \cr 0 \end{pmatrix},
        \end{gather}
then the Weyl function $M(\cd)$ can also be defined as the unique operator-valued function
$M(\cd):\rho (\wt A)\to \B (\C^p)$ such that
        \begin{gather}\label{3.24.1}
\ff(\cd,\l)M(\l)+\psi (\cd,\l)\in L^2(\R_+,\B (\C^p,\C^n)), \quad \l\in\rho (\wt A).
        \end{gather}
\item
If $\l\in\rho (\wt A)$ and $\wh v(\cd,\l)=(\wh v_1(\cd,\l), \wh v_2(\cd,\l))^\top \in L^2(\R_+,\B (\C^p,\C^{2p}))$ is an operator solution of \eqref{3.22}, then the operator $C_1 v_1(0,\l)+C_2 v_2(0,\l)$ is invertible and
       \begin{gather*}
M(\l)=(C_3 v_1(0,\l)+C_4 v_2(0,\l))(C_1 v_1(0,\l)+C_2 v_2(0,\l))^{-1}.
      \end{gather*}
\end{enumerate}
\end{theorem}
\begin{proof}
{\rm (i)} Let $\wt \G^B=(\wt \G_0^B, \wt \G_1^B)^\top$ and $\wt \G^A=(\wt \G_0^A, \wt \G_1^A)^\top$ be  operators defined by
        \begin{gather}
\wt \G_0^B y=y_1(0), \quad  \wt \G_1^B y=y_2(0), \quad y\in\dom B^*,\label{3.25.1}\\
\wt \G_0^A z=z_1(0), \quad  \wt \G_1^A z=z_2(0), \quad z\in\dom A^*.\label{3.25.2}
        \end{gather}
It follows from \eqref{3.11.0} and the Lagrange identity \eqref{3.2} that these operators satisfy \eqref{3.14}. Moreover, by Lemma \ref{lem3.2} the operators $\G^B$ and $\G^A$ are surjective. Therefore the collection $\wt\Pi=\{\C^p\oplus\C^p,\wt\G^B,\wt\G^A\}$ is a boundary triple for $\{A,B\}$.

Since the pair $C_1,C_2$ is admissible, there exist operators $C_3,C_4 \in \B (\C^p)$ and $C_j'\in \B (\C^p), \; j\in \{1,\dots,4\},$ such that the equalities \eqref{3.17} define invertible operators $X$ and $Y$ satisfying $Y^* J X=J$. Moreover,  according to \cite[Proposition 4.6]{MM02} for each such  pair the equalities
        \begin{gather*}
\G^B= \begin{pmatrix} \G_0^B\cr \G_1^B    \end{pmatrix}:= \begin{pmatrix} C_1&C_2\cr C_3&C_4\end{pmatrix}\begin{pmatrix} \wt \G_0^B\cr \wt\G_1^B    \end{pmatrix}, \quad \G^A= \begin{pmatrix} \G_0^A\cr \G_1^A    \end{pmatrix}:= \begin{pmatrix} C_1'&C_2'\cr C_3'&C_4'\end{pmatrix}\begin{pmatrix} \wt \G_0^A\cr \wt\G_1^A    \end{pmatrix}
         \end{gather*}
define a boundary triple $\Pi=\{\C^p\oplus\C^p,\G^B,\G^A\}$ for the dual pair $\{A,B\}$
and in view of \eqref{3.25.1}, \eqref{3.25.2}, the mappings  $\G^B$ and $\G^A$ are of
the form \eqref{3.18} -- \eqref{3.20}. Moreover, the equality $A_0=\wt A$ is implied by
combining \eqref{3.11.2}  with  \eqref{3.18}.

{\rm (ii)} Let $\g(\cd) $ be the $\g$-field \eqref{3.15}  of the triple $\Pi$. Since
$\gN_\l (B)$ coincides with the space of all $L^2(\R_+,\C^n)$-solutions  of equation
\eqref{3.22}, the equality
        \begin{gather}\label{3.26}
v(t,\l)h=(\g(\l)h)(t), \qquad h\in\C^p,\quad  \l\in\rho (\wt A)
        \end{gather}
defines a unique operator solution $v(t,\l)(\in  \B (\C^p,\C^n))$ of \eqref{3.22} such
that $v(\cd,\l)h\in L^2(\R_+,\C^n)$ and
        \begin{gather}\label{3.27}
\G_0^B v(\cd,\l)h=h, \quad h\in\C^p,\;\;\l\in\rho (\wt A).
        \end{gather}
Moreover, it follows from \eqref{3.18} and the block representation \eqref{3.21} of
$v(\cd,\l)$ that \eqref{3.27} is equivalent to \eqref{3.23}. The holomorphy of
$v(\cd,\l)$ in $\rho (\wt A)$ is immediate  from \eqref{3.26} since the $\g$-field
$\g(\cd)$ is holomorphic on $\rho (\wt A)$.

{\rm (iii)} Recall that  in accordance with  \eqref{3.15}, $M(\l)=\G_1^B \g(\l)$.
Therefore  it follows from \eqref{3.26} and \eqref{3.18} that \eqref{3.23.1} holds.
Next, combining \eqref{3.23}  with  \eqref{3.23.1} yields
        \begin{gather*}
Xv(0,\l)= \begin{pmatrix} I_p\cr M(\l)   \end{pmatrix}=X(\ff(0,\l)M(\l)+\psi (0,\l)).
        \end{gather*}
Since $X$ is invertible, one gets  $v(t,\l)=\ff(t,\l)M(\l)+\psi (t,\l)$. This proves the
second statement in  (iii).

(iv)  This statement  is immediate by combining  (ii) with  (iii).
\end{proof}
The following theorem is an extended version
%%which covers
of Theorem \ref{Th1.2_Intro} from the Introduction.  In particular, we prove here the
existence of a holomorphic Weyl solution for  a $j$-symmetric Dirac type operator on the
half-line.
  \begin{theorem}\label{th3.9.1}
Let $n=2p$, let $D(Q)$ be   the Dirac type expression \eqref{3.1} on $\R_+$ with $J_n$ and $Q(\cd)$ given by \eqref{3.1.5}, let   the relations \eqref{3.1.6} be fulfilled and let $j$ be the conjugation in $L^2(\R_+,\C^n)$ given by \eqref{3.1.4}. If $\wh \rho (\Dmi^+(Q))\neq\emptyset$, then:
   \begin{enumerate}\def\labelenumi{\rm (\roman{enumi})}
\item
The minimal operator $\Dmi^+(Q)$ is $j$-symmetric and ${\rm Def}(\Dmi^+(Q))=p,$ i.e., for each $\l\in\wh \rho (\Dmi^+(Q))$ (or, equivalently, for each $\l\in\wh\rho \{\Dmi^+(Q), \Dmi^+(Q^*)\}$) the following holds:
        \begin{gather}\label{3.28}
\dim\gN_\l(\Dmi^+(Q^*))=\dim\gN_{\ov\l}(\Dmi^+ (Q))= p.
         \end{gather}
\item
There exists an operator $C\in\B (\C^p,\C^n)$ such that the equalities
 \begin{equation*}
\wt A=\Dma^+ (Q) \upharpoonright \dom\wt A, \quad  \dom \wt A =\{y\in\dom\Dma^+ (Q): C
y=0\},
  \end{equation*}
define  a $j$-selfadjoint extension $\wt A$ of $\Dmi^+(Q)$ with $\rho (\wt
A)\neq\emptyset$ and for each $\l\in \rho(\wt A)$ there exists a unique operator-valued
solution  $\Psi_+(\cdot,\l) \in L^2(\R_+,\B (\C^p,\C^n))$  (the Weyl solution) of
\eqref{3.22} satisfying $C\Psi_+(0,\l) = I_p$. Moreover, the solution $\Psi_+(\cd,\l)$
is holomorphic in $ \rho (\wt A)$.
\end{enumerate}
\end{theorem}
\begin{proof}
(i)  As in Proposition \ref{pr3.0.1} (see formula \eqref{3.1.7})  the operator $\Dmi^+(Q)$ is
$j$-symmetric and
        \begin{gather*}
j(\Dmi^+ (Q))^*j = j\Dma^+(Q^*)j=\Dma^+(Q).
         \end{gather*}
Combining this formula with   \eqref{2.55}  and Corollary \ref{cor3.3.0} (see formula
\eqref{3.11a}) gives
  \begin{eqnarray}
2 \text{Def} D_{\min}^+(Q) = \dim (\dom (jD_{\min }^+(Q)^*j)/\dom D_{\min}^+(Q)) \nonumber \\
= \dim (\dom D_{\max}^+(Q)/\dom D_{\min}^+(Q)) = n = 2p.
 \end{eqnarray}
%
%
%%        \begin{gather*}
%%2 {\rm Def}(\Dmi (Q)) =\dim (\dom\Dma(Q)/\dom\Dmi(Q))=n=2p,
%%        \end{gather*}
This  proves the statement.

(ii) According to Proposition \ref{prop_gap_preserving} there exists a $j$-selfadjoint
extension $\wt A$ of $\Dmi^+(Q)$ such that $\rho (\wt A)\neq\emptyset$. Since  $\wt A$ is a quasi-selfadjoint extension, the statement (ii) follows from Corollary \ref{cor3.5} and Theorem~\ref{th3.9}(ii).
     \end{proof}
%%%%%%%%%%%%%%%%%%%%%%%%%%%%%%%%%%%%%%%%%%%%%%%
\begin{corollary}
Assume the conditions of  Theorem \ref{th3.9.1}. Let  $\l_0\in\wh\rho (\Dmi^+(Q))$ and let
$\D(\l_0;\varepsilon)$ be  a gap for the operator $\Dmi^+(Q)$.  Then there is an operator
$C\in\B (\C^p,\C^n)$ such that the corresponding  Weyl solution $\Psi_+(\cd,\l)$ is
holomorphic in  $\D(\l_0;\varepsilon)$.
\end{corollary}
  \begin{proof}
According to Proposition \ref{prop_gap_preserving} there exists a $j$-selfadjoint
extension $\wt A$ of $\Dmi^+(Q)$ preserving the gap $\D(\l_0;\varepsilon)$. Starting with
this extension it remains to repeat the reasoning of Theorem \ref{th3.9.1}(ii).
   \end{proof}

\begin{remark}\label{rem3.6.2}
\begin{enumerate}\def\labelenumi{\rm (\roman{enumi})}

\item  In the scalar case $(p=1)$ and a Dirac type expression $D(Q)$ on $\R_+$ with
$Q(\cdot)$ of the form \eqref{3.1.10} and arbitrary   scalar $q(\cdot)$, Theorem
\ref{th3.9.1} was proved by another method in \cite{CGHL04} (see also \cite{CG06}). Note
also the paper \cite{Kla10}, where equality \eqref{3.28} (without a connection with the
Weyl type solutions) was proved for a Dirac type expression \eqref{3.1} with a special
off-diagonal potential matrix $Q(\cd)$ such that \eqref{3.1.7} is satisfied.

\item   Clearly, a  similar result is valid for  the  Dirac system  on $\R_-$ with a
potential matrix  $Q\in L_{\rm loc}^1(\R_-,\C^{n\times n})$, i.e. there exists the
unique $\B(\C^p,\C^n)$-valued Weyl solution $\Psi_-(\cdot,\l) \in L^2(\R_-,\B
(\C^p,\C^n))$ .

Note also, that    in the scalar case $(p=1)$ for a Dirac type  operator  on the line with
a potential matrix  $Q (\in L_{\rm loc}^1(\R,\C^{2\times 2}))$ of the form
\eqref{3.1.10} a  stronger  result  is proved in   \cite[Theorem 5.4]{CGHL04}, namely %a stronger result:
  \begin{equation}
\sup_{x\in \R}\left[\left(\int^x_{-\infty}\|\Psi_-(t,z)\|^2_{\mathbb{C}^2} \right) \cdot
\left(\int^{\infty}_x\|\Psi_+(t,z)\|^2_{\mathbb{C}^2}\,dx\right)\right] < \infty.
 \end{equation}
 \end{enumerate}
\end{remark}

\subsection{The Weyl  function of  almost formally selfadjoint Dirac type  operators  }

Clearly, the Dirac type expression \eqref{3.1} admits the representation
        \begin{gather}\label{3.30}
D(Q)y= J_{n} \frac{dy}{dx} + Q_1(x)y+i Q_2(x)y,
        \end{gather}
where $Q_1(x)=\overline{Q_1(x)}:={\rm Re}\, Q(x) $ and $Q_2(x)=\overline{Q_2(x)}:=\im
Q(x), \; x\in\R_+ $.
   \begin{definition}\label{def3.10}
The Dirac type expression \eqref{3.30} will be called almost  formally selfadjoint if the operator-function
$Q_2(\cd)$ is bounded, that is $||Q_2(x)||\leq C, \; x\in\R_+,$ with some $C\in\R$.
  \end{definition}
Let the  expression \eqref{3.30} be almost selfadjoint and let $\a (x)=\inf\limits_{h\in\C^n}
\frac {(Q_2(x)h,h)}{||h||^2}$ and  $\b(x)=\sup\limits_{h\in\C^n} \frac
{(Q_2(x)h,h)}{||h||^2}$ be lower and upper bounds   respectively of $Q_2(x)$. In the following we
denote by $\a_Q$ and $\b_Q$   the  essential infimum and supremum of $Q_2(x)$ respectively:
        \begin{gather*}
\a_Q={\rm ess} \inf_{x\in\R_+}Q_2(x), \qquad  \b_Q={\rm ess} \sup_{x\in\R_+}Q_2(x).
        \end{gather*}
  \begin{proposition}\label{pr3.11}
Let expression \eqref{3.30} on $\R_+$ be almost formally selfadjoint, let $\kappa_\pm=\dim \ker (J_n\pm i
I_n)$ and let $A:=\Dmi^+(Q), \; B:=\Dmi^+(Q^*)$. Then $\dom A=\dom B=:\cD$, $\dom A^*=\dom
B^*$ and
        \begin{gather*}
\a_Q ||f||^2\leq \im (Af,f)\leq  \b_Q ||f||^2, \qquad - \b_Q ||f||^2\leq\im (Bf,f)\leq -\a_Q ||f||^2, \quad f\in\cD.
        \end{gather*}
Moreover, $(\C_{\a_Q,-}\cup \C_{\b_Q,+})\subset \wh \rho \{A,B\}$ and
\begin{gather*}
\dim \gN_\l(B)=\kappa_+, \qquad \dim \gN_{\overline\l}(A) =\kappa_-,\quad \l\in\C_{\b_Q,+},\\
\dim \gN_\l(B)=\kappa_-, \qquad \dim \gN_{\overline\l}(A) =\kappa_+,\quad \l\in\C_{\a_Q,-}.
 \end{gather*}
\end{proposition}
\begin{proof}
Let $S:=\Dmi^+ (Q_1)$ be the minimal operator  generated in $\gH$ by the formally selfadjoint  Dirac type expression
       \begin{gather}\label{3.31}
D(Q_1)y= J_{n} \frac{dy}{dx} + Q_1(x)y
       \end{gather}
and let $T$ be the multiplication operator in $\gH$ defined by
       \begin{gather}\label{3.31.0}
(Tf)(x)=Q_2(x)f(x), \quad f(\cdot)\in\gH.
       \end{gather}
Clearly, $S$ is a closed densely defined symmetric operator in $\gH$,
and in accordance with  \cite[Theorem 5.2]{LM03} $n_\pm(S)=\kappa_\pm$.
Moreover,  $T=T^*\in\B (\gH)$ and the lower and upper bounds
of $T$ are $\a_Q$ and $\b_Q$ respectively. Therefore  \eqref{3.30} yields
    \begin{gather}\label{3.31.1}
A=S+iT, \qquad B=S-iT.
       \end{gather}
It remains to apply  Lemma \ref{lem3.8.2}.
\end{proof}

\begin{theorem}\label{th3.12}
Assume that $n=2p$, expression \eqref{3.30} on $\R_+$ is almost formally selfadjoint and $J_n$ is of the form \eqref{3.11.0}. Then:
\begin{enumerate}\def\labelenumi{\rm (\roman{enumi})}
\item
For each operator $\Phi=\Phi^*\in\B (\C^p)$ the equalities (the boundary conditions)
            \begin{equation*}
\dom \wt A_\Phi=\{y\in\dom\Dma^+(Q): \cos \Phi\cdot y_1(0)+\sin \Phi \cdot y_2(0)=0 \}, \quad \wt A_\Phi=\Dma^+ (Q)\up \dom \wt A_\Phi
         \end{equation*}
define a quasi-selfadjoint extension $\wt A_\Phi\in \Ext \{\Dmi^+(Q),\Dmi^+ (Q^*)\}$ satisfying
           \begin{gather}\label{3.33}
\a_Q ||f||^2\leq \im (\wt A_\Phi f,f)\leq  \b_Q ||f||^2, \;\; f\in\dom \wt A_\Phi ,\;\;\;{\rm and} \;\;\;  (\C_{\a_Q,-}\cup \C_{\b_Q,+} )\subset \rho (\wt A_\Phi).
        \end{gather}
 \item
For each $\l\in \rho(\wt A_\Phi)$ there exists a unique operator solution
        \begin{gather}\label{3.34}
v_\Phi(t,\l)=(v_{1,\Phi}(t,\l), v_{2,\Phi}(t,\l))^\top: \C^p\to \C^p\oplus\C^p, \quad t\in\R_+
        \end{gather}
of the homogeneous equation \eqref{3.22} such that $v_\Phi(\cd,\l)h\in L^2(\R_+,\C^n), \; h\in \C^p,$ and
        \begin{gather}\label{3.35}
\cos\Phi\cdot v_{1,\Phi}(0,\l)+ \sin\Phi\cdot v_{2,\Phi}(0,\l)=I_p, \quad \l\in \rho(\wt A_\Phi).
        \end{gather}
\end{enumerate}
\end{theorem}
\begin{proof}
Let $A:=\Dmi^+(Q), \;B:=\Dmi^+(Q^*) $, let $S$ and $T$ be the same as in the proof of Proposition \ref{pr3.11} and let $\cD_*:=\dom S^*$. Then by \eqref{3.31.1} and Lemma \ref{lem3.8.2}, (iii) the equalities  \eqref{3.15.3} hold.

Next, assume that $\Phi=\Phi^*\in \B (\C^p)$. Then the operator
        \begin{equation}\label{3.36}
X=\begin{pmatrix}  \cos\Phi & \sin\Phi \cr -\sin\Phi& \cos\Phi\end{pmatrix}:\C^p\oplus\C^p\to \C^p\oplus\C^p
        \end{equation}
satisfies $X^* J_n X=J_n$ and  by Theorem \ref{th3.9}, (i) applied to the  Dirac type expression \eqref{3.31} the equalities
\begin{gather}\label{3.36.1}
\G_0 y= \cos\Phi\cdot y_1(0) + \sin\Phi\cdot y_2(0), \quad   \G_1 y= -\sin\Phi\cdot y_1(0) +\cos\Phi\cdot  y_2(0), \quad y\in\cD_*
        \end{gather}
define a boundary triple $\wh\Pi=\{\C^p, \G_0,\G_1\} $ for $S^*$. Moreover, application of Theorem \ref{th3.9}, (i) to the dual pair $\{A,B\}$   implies that a collection $\Pi=\{\C^p\oplus\C^p,\G^B,\G^A\}$ with operators $\G^B=(\G_0^B,\G_1^B)^\top$ and $\G^A=(\G_0^A,\G_1^A)^\top$ defined by
         \begin{gather}
\G_0^B y=\G_0^A y=\cos \Phi\cdot y_1(0) + \sin\Phi \cdot y_2(0),\label{3.37} \\
\G_1^B y= \G_1^A y=-\sin\Phi \cdot y_1(0) +\cos \Phi \cdot y_2(0), \quad y\in\cD_*\label{3.37.1}
        \end{gather}
is a boundary triple for $\{A,B\}$ and for this triple  $A_0(=B^*\up \ker \G_0^B)=\wt A_\Phi$. Since the triples $\Pi$ and $\wh\Pi$ are connected via \eqref{3.15.4}, it follows from Theorem \ref{th3.8.3}, (i) that \eqref{3.33} holds.

Statement (ii) directly follows from Theorem \ref{th3.9}, (ii) applied to the extension $\wt A=\wt A_\Phi$.
 \end{proof}
Let the assumptions of Theorem \ref{th3.12} be satisfied and let $\Phi=\Phi^*\in \B (\C^p)$.
\begin{definition}\label{def3.13}
The operator solution $v(\cd,\l)=v_\Phi(\cd,\l),\; \l\in\rho (\wt A_\Phi),$ of \eqref{3.22} defined in Theorem \ref{th3.12} will be called the Weyl solution (with respect to the parameter $\Phi$).
\end{definition}
\begin{definition}\label{def3.14}
The operator function $M(\cd)=M_\Phi(\cd):\rho (\wt A_\Phi)\to  \B (\C^p)$ defined by
         \begin{gather}\label{3.38}
M_\Phi(\l)= -\sin\Phi\cdot v_{1,\Phi}(0,\l)+  \cos \Phi\cdot   v_{2,\Phi}(0,\l), \quad \l\in \rho (\wt A_\Phi)
         \end{gather}
will be called the Weyl function of  the  almost formally-selfadjoint Dirac type expression \eqref{3.30} on $\R_+$ (with respect to the parameter $\Phi$).
\end{definition}

\begin{proposition}\label{pr3.15}
Assume the hypothesis of Theorem \ref{th3.12}. Let $\Phi=\Phi^*\in \B (\C^p)$ and let $\ff(\cd,\l)$ and $\psi (\cd,\l)$ be  $\B (\C^p, \C^p\oplus\C^p)$-valued solutions of \eqref{3.22} with the initial values
         \begin{gather}\label{3.38.1}
\ff(0,\l)= \begin{pmatrix}  -\sin\Phi\cr \cos\Phi \end{pmatrix}, \qquad \psi(0,\l)= \begin{pmatrix}  \cos\Phi \cr \sin\Phi \end{pmatrix}.
        \end{gather}
Then:
\begin{enumerate}\def\labelenumi{\rm (\roman{enumi})}
\item
the Weyl function $M_\Phi(\cd)$ of  expression \eqref{3.30} can be defined as a unique operator-function $M_\Phi(\cd):\rho (\wt A_\Phi)\to \B (\C^p)$ such that \eqref{3.24.1} holds, i.e.,
         \begin{gather*}
\ff(\cd,\l)M(\l)+\psi (\cd,\l)\in L^2(\R_+,\B (\C^p,\C^n)), \quad \l\in\rho (\wt A_\Phi);
         \end{gather*}
\item
$M_\Phi(\cd)$ is holomorphic in $\rho (\wt A_\Phi)$  and satisfies the identities
         \begin{gather}
M_\Phi(\l)- M_\Phi^*(\mu)=\int_{\R_+}v_\Phi^*(x,\mu)((\l-\overline \mu)I_n-2i Q_2(x))v_\Phi(x,\l)\,dx, \quad \l,\mu\in \rho (\wt A_\Phi),\label{3.39}\\
\im M_\Phi(\l)=\int_{\R_+}v_\Phi^*(x,\l)(\im\l\cdot I_n- Q_2(x) )v_\Phi(x,\l)\,dx, \quad \l\in \rho (\wt A_\Phi).\label{3.40}
        \end{gather}
Moreover,
      \begin{gather}\label{3.41}
\im M_\Phi(\l)\geq \delta_\l I, \quad \l\in\C_{\b_Q,+};  \qquad  \im M_\Phi(\l)\leq \delta_\l' I, \quad \l\in\C_{\a_Q,-}
       \end{gather}
with some $\delta_\l>0$ and $\delta_\l'<0$ (depending on $\l$).
\end{enumerate}
\end{proposition}

\begin{proof}
Let $A:=\Dmi^+ (Q), \; B:=\Dmi^+ (Q^*)$ and let $X$ be the   operator \eqref{3.36}. It was shown in the proof of Theorem \ref{th3.12} that the operators  $\G^B=(\G_0^B,\G_1^B)^\top$ and $\G^A=(\G_0^A,\G_1^A)^\top$ defined by
\eqref{3.37} and \eqref{3.37.1} form a boundary triple  $\Pi=\{\C^p\oplus\C^p,\G^B,\G^A\}$ for $\{A,B\}$. Since by \eqref{3.36} the operators $C_3$ and $C_4$ in \eqref{3.23.1} are $C_3=-\sin\Phi$ and $C_4=\cos\Phi$, it follows from Theorem \ref{th3.9}, (iii) and \eqref{3.38} that the Weyl function $M(\cd)$ of $\Pi$ coincides with $M_\Phi(\cd)$. Moreover,
        \begin{gather*}
X^{-1}=\begin{pmatrix}  \cos\Phi & -\sin\Phi \cr \sin\Phi& \cos\Phi\end{pmatrix}
        \end{gather*}
and hence the initial conditions \eqref{3.24} are equivalent to \eqref{3.38.1}. Now statement (i) follows from Theorem \ref{th3.9}, (iii).

Next, assume that $S$ and $T$ are the same as in the proof of Proposition \ref{pr3.11} and let $\wh\Pi=\{\C^p\oplus\C^p,\G_0,\G_1\}$ be the boundary triple \eqref{3.36.1} for $S^*$. Since $A$ and $B$ admit  representation  \eqref{3.31.1} and boundary triples $\Pi$ and $\wh\Pi$ are connected via \eqref{3.15.4}, it follows from Theorem \ref{th3.8.3} that identities \eqref{3.15.6} and \eqref{3.15.7} hold with $M(\l)=M_\Phi(\l)$ and $A_0=\wt A_\Phi$. By using \eqref{3.26} one can easily verify that
        \begin{gather*}
\g^*(\l)f(\cdot)=\int_{\R_+} v^*(t,\l)f(t)\, dt, \quad f(\cdot) \in \gH, \; \l\in\rho (\wt A_\Phi).
        \end{gather*}
This and \eqref{3.26}, \eqref{3.31.0} imply that \eqref{3.15.6} and \eqref{3.15.7} can be written in the form \eqref{3.39} and \eqref{3.40}, respectively. Finally, \eqref{3.41} is immediate from \eqref{3.15.8} and the equality $M(\l)=M_\Phi(\l), \; \l\in \C_{\b_Q,+}\cap \C_{\a_Q,-}$.
\end{proof}
\begin{corollary}\label{cor3.16}
Let $n=2p$,   and  let $D(Q) $ be an almost formally-selfadjoint  Dirac type expression \eqref{3.1} (or, equivalently, \eqref{3.30}) on $\R_+$
with
      \begin{gather}\label{3.44}
      J_n=\begin{pmatrix} -i I_p&0\cr 0 & iI_p \end{pmatrix}: \C^p\oplus\C^p\to \C^p\oplus\C^p
      \end{gather}
and let $\ff(\cdot, \l)$ and $\psi (\cdot,\l)$
be $\B (\C^p,\C^p\oplus\C^p)$-valued solution of \eqref{3.22} with the initial values
        \begin{gather*}
\ff(0,\l)=\begin{pmatrix}0\cr I_p \end{pmatrix}, \qquad \psi(0,\l)=\begin{pmatrix} I_p\cr 0 \end{pmatrix}.
        \end{gather*}
Then there exists a unique operator-function $M_s(\cd):\C_{\b_Q,+}\to \B (\C^p)$ such that
      \begin{gather}\label{3.45}
\ff(\cd,\l)M_s(\l)+\psi (\cd,\l)\in L^2(\R_+,\B (\C^p,\C^n)), \quad \l\in\C_{\b_Q,+} .
      \end{gather}
Moreover, $M_s(\cd)$ is holomorphic on $\C_{\b_Q,+}$ and satisfies $||M_s(\l)||<1, \; \l\in \C_{\b_Q,+}$.
\end{corollary}
\begin{proof}
One can easily verify that the equality $U=\frac 1 {\sqrt 2} \begin{pmatrix}  i I_p & I_p\cr -i I_p & I_p      \end{pmatrix}$
defines a unitary operator $U\in\B (\C^n)$ such that
        \begin{gather*}
 \wh J_n:=U^* J_n U=\begin{pmatrix} 0 & - I_p \cr I_p & 0     \end{pmatrix}.
         \end{gather*}
Let $\wh Q(x):=U^*Q(x)U=\wh Q_1(x)+i \wh Q_2(x)$ with $\wh Q_1(x)={\rm Re}\, \wh Q(x)$ and $\wh Q_2(x)=\im \wh Q(x)$, let
        \begin{gather*}
D(\wh Q)y:=\wh  J_{n} \frac{dy}{dx} + \wh Q_1(x)y+i\wh Q_2(x)y ,\quad x\in\R_+
        \end{gather*}
be the respective almost selfadjoint Dirac type expression and let $M_0(\cd)$ be the Weyl function of $D(\wh Q)$ (with respect to $\Phi=0$). Since obviously $\b_{\wh Q}=\b_Q$, it follows that $M_0(\cd)$ is defined on $\C_{\b_Q,+}$ and by \eqref{3.41} the equality (the Cayley transform of $M_0(\l)$)
        \begin{gather*}
M_s(\l)=(M_0(\l)-i I_p)(M_0(\l)+i I_p)^{-1}, \quad \l\in  \C_{\b_Q,+}
       \end{gather*}
defines a holomorphic operator-function $M_s(\cd):\C_{\b_Q,+}\to \B (\C^p)$ such that $||M_s(\l)||<1, \; \l\in \C_{\b_Q,+}$. Let $v_0(\cd,\l)\in L^2(\R_+,\B (\C^p,\C^n))$ be the Weyl solution of the homogeneous equation $D(\wh Q)y=\l y$ (with respect to $\Phi=0$) and let
          \begin{gather*}
v_s(x,\l)=U v_0(x,\l)\sqrt 2 (M_0(\l)+i I_p)^{-1}, \quad x\in\R_+, \;\; \l\in  \C_{\b_Q,+}.
        \end{gather*}
Clearly, $v_s(\cd,\l)$ is a solution of \eqref{3.22} and $v_s(\cd,\l)\in L^2(\R_+,\B (\C^p,\C^n))$. Moreover, since $\Phi=0$, it follows from \eqref{3.35} and \eqref{3.38} that $v_0(0,\l)=(I_p, M_0(\l))^\top$ and hence $v_s(0,\l)=(I_p, M_s(\l))^\top$. Therefore $v_s(x,\l)=\ff(x,\l)M_s(\l)+\psi (x,\l)$, which implies \eqref{3.45}.
\end{proof}
\begin{remark}\label{rem3.17}
{\rm (i)} Let $D(Q)$ be  an  almost formally-selfadjoint Dirac type expression \eqref{3.1} with $J_n$ of the form \eqref{3.44}, let $A:=\Dmi^+(Q), \; B:=\Dmi^+ (Q^*)$ and let $D_*=\dom A^*=\dom B^*$ (see Proposition \ref{pr3.11}). It follows from the Lagrange identity \eqref{3.2} and Lemma \ref{lem3.2} that a collection $\Pi_s=\{\C^p\oplus \C^p,\G^B,\G^A\}$ with operators $\G^B=(\G_0^B,\G_1^B )^\top$ and $\G^A=(\G_0^A,\G_1^A )^\top$ defined by
           \begin{gather*}
\G_0^B y=y_1(0), \quad  \G_1^B y=y_2(0), \quad   \G_0^A y= i y_2(0),\quad \G_1^A y= i y_1(0), \;\; y\in \cD_*
           \end{gather*}
is a boundary triple for $\{A,B\}$ (here $y_j(0)$ are taken from \eqref{3.11.1}). Moreover, one can easily prove that $M_s(\cd)$ is the Weyl function of $\Pi_s$.

(ii) The operator-function $M_s(\cdot)$ coincides with the Weyl function introduced in \cite{FKRS12,SSR13} for special almost formally-selfadjoint Dirac type expression \eqref{3.1}, \eqref{3.44} with $Q(x)$ of the form
           \begin{gather*}
Q(x)=i\begin{pmatrix} 0 & q(x) \cr q^*(x) & 0  \end{pmatrix}:\C^p\oplus\C^p \to \C^p\oplus\C^p.
           \end{gather*}
 \end{remark}

{\bf Acknowledgement}
The publication was prepared with the support of the RUDN
University Program 5-100.  BMB, MK, MM, IGW  thank the Isaac Newton Institute at Cambridge  for support to attend  the  Inverse problems  programme held there in 2011  where much of  the initial work was undertaken.

\end{document}